\documentclass[12pt,a4paper]{article}

\usepackage{amssymb, amsmath, amsthm}
\usepackage{pgfplots}
\pgfplotsset{compat=1.18}

\usepackage[cm]{fullpage}
\usepackage[english]{babel}
\usepackage[T1]{fontenc}
\usepackage{graphicx}
\usepackage{booktabs}
\usepackage{xcolor}

\usepackage{tikz}
\usetikzlibrary{arrows.meta,decorations.pathreplacing,calc,positioning,patterns,fit}

\usepackage{float}
\DeclareGraphicsExtensions{.png,.pdf}
\graphicspath{{./images/}}

\usepackage{hyperref}
\usepackage{autonum}

\newtheorem{theorem}{Theorem}

\newtheorem{definition}{Definition}

\title{A Unified Model for Blood and Lymph Flow with Coupled Nonsmooth Biochemical Dynamics}
\author{
Bogna Jaszczak-Dyka$^*$, \and 
{\L}ukasz P{\l}ociniczak\thanks{Faculty of Pure and Applied Mathematics, Wroclaw University of Science and Technology, Wyb. Wyspia\'nskiego 27, 50-370 Wroc{\l}aw, Poland, emails: \texttt{bogna.jaszczak@pwr.edu.pl}, \texttt{lukasz.plociniczak@pwr.edu.pl}}
}
\date{}

\begin{document}
\maketitle

\begin{abstract}

We present a unified mathematical framework for modeling blood and lymph flow in biological vessels, with a particular focus on lymph transport through lymphangions. Starting from first principles, we rigorously derive a system of partial differential equations (PDEs) that govern the fluid dynamics using perturbative methods. To capture the active regulation of lymphangion valves, we couple these PDEs with a system of two nonlinear ordinary non-smooth differential equations (ODEs) describing the chemical kinetics of calcium ions ($Ca^{2+}$) and nitric oxide ($NO$). These biochemical species play a critical role in valve opening and closing, influencing lymph propulsion. We further analyze a reduced model consisting of two non-smooth ODEs, identifying parameter regimes that guarantee the existence of a stable limit cycle. This oscillatory behavior aligns with experimental observations of lymphatic pumping, providing theoretical validation and new insights into lymphatic physiology. Our results offer a comprehensive mathematical description of lymph flow regulation and open possibilities for future studies on pathological conditions and therapeutic interventions.\\
    

\end{abstract}

\section{Introduction}
The lymphatic system constitutes a vital network of vessels, nodes, and lymphoid organs that maintains tissue fluid homeostasis by returning excess interstitial fluid and macromolecules to the venous circulation, while simultaneously supporting the immune system and lipid transport \cite{choi2012new}. Beyond its classical role in preventing edema and maintaining optimal tissue perfusion, lymphatic function is now recognized as critical in inflammation and the pathophysiology of major conditions, including cardiovascular disease, cancer, obesity, and autoimmunity \cite{escobedo2017lymphatic,mehrara2023emerging}. Despite these fundamental contributions to human health, the lymphatic system has historically been overshadowed by the high-pressure cardiovascular system. Although described alongside blood vessels as early as Hippocrates \cite{grotte1979discovery} (or even as early as 1600 BC in one of the Ancient Egypt hieroglyphs that can be translated as "lymphatic glands swelling" - see the complete historical account in \cite{van2022lymphatic}), the lymphatic system has long been regarded as secondary or invisible and described by some as "arguably the most neglected bodily system" (see \cite{mortimer2014new}) resulting in a scarcity of research until recent molecular and imaging advances triggered a renaissance in the field \cite{mehrara2023emerging,choi2012new,mills2024magnetic}. However, compared to the extensive literature on hemodynamics, mathematical and computational studies of lymph flow remain relatively scarce, motivating the development of dedicated models to describe the complex mechanics of lymphatic transport in both healthy and sick individuals \cite{jayathungage2024computational}.

From a fluid dynamics perspective, the lymphatic system functions as a hierarchical transport network designed to pump fluid against a net pressure gradient from the low-pressure interstitial space to the higher-pressure venous circulation. The network topology begins with the initial lymphatics, which act as blind-ended porous capillaries \cite{Null2025AnatomyLymphaticSystem}. Here, fluid absorption is governed primarily by local transmural pressure differences and mechanical coupling to the surrounding tissue \cite{mchale1976effect}. These capillaries converge into collecting vessels that are structurally segmented into a series of contractile chambers known as lymphangions \cite{wilting2022lymphatic}. Each lymphangion is bounded by non-return valves that ensure unidirectional flow by preventing retrograde motion. Unlike the passive venous system, the transport in collecting vessels is actively driven by the rhythmic contraction and relaxation of the vessel walls that function as a biological pump that generates the necessary pressure to propel the lymph \cite{munn2015mechanobiology, kunert2015mechanobiological}. A substantial malfunction of the pumping mechanism has a profound effect on the overall health of the surrounding tissue and can lead to lymphedema or immune dysfunction \cite{liao2011impaired}. From a mathematical perspective, the lymphatic network can thus be viewed as a distributed, actively pumped conduit system with heterogeneity in geometry, wall mechanics, and valvular structure, all of which shape the spatio-temporal patterns of lymph flow and pressure that models aim to capture.

To model lymphatic transport, it is important to distinguish its physical properties from those of the cardiovascular system, as these differences influence the choice of governing equations and boundary conditions. Physically and hemodynamically, the lymphatic system differs fundamentally from the arterial and venous circulations. While the cardiovascular system forms a closed cycle driven by a central pump (the heart), the lymphatic system is an open linear network that transports fluid from the interstitial space to the central veins \cite{alitalo2002molecular, jafarnejad2015modeling}. The most essential distinction lies in the pressure regimes and flow mechanisms: arterial flow is high-pressure and pulsatile, driven by cardiac systole, and venous flow is lower-pressure and quasi-steady. On the other hand, lymphatic flow is low-pressure (often sub-atmospheric in initial vessels), pulsatile, intermittent, and oscillatory, driven by intrinsic contractions and stochastic external compression \cite{moore2018lymphatic}. Structurally, unlike the thick elastic walls of arteries designed to withstand high shear stresses, lymphatic vessels possess thin, highly compliant walls that are coupled to the extracellular matrix, making them sensitive to deformation \cite{breslin2019lymphatic}. Furthermore, both fluids differ substantially. Although blood is a non-Newtonian suspension of red blood cells that causes shear-thinning flow \cite{alexy2022physical}, lymph is generally a protein-rich fluid with a much lower concentration of cells (primarily lymphocytes), often allowing the approximation of Newtonian behavior in larger collecting vessels \cite{santambrogio2018lymphatic}. For readers convenience, we collect some typical physical parameters of the blood and lymph in Tab. \ref{tab:BloodLymph}.

As mentioned above, the lymphatic system does not have its own external pumping mechanism and the cardiovascular system. Lymph flow through the lymphangion is driven by contractions of lymphatic muscle cells, which generate synchronized pressure pulses to open distal valves and eject fluid while closing the proximal valves to prevent backward flow. This pulsatile mechanism is governed by a complex mechano-chemical oscillator that couples intracellular calcium ($Ca^{2+}$) concentration with transients of nitric oxide (NO) \cite{munn2015mechanobiology}. The contractile phase (systole) is driven by voltage-gated $Ca^{2+}$ entry during depolarization, which triggers actin-myosin cross-bridge cycling and vessel constriction \cite{scallan2016lymphatic}. On the other hand, the relaxation phase (diastole) is actively modulated by a flow-dependent negative feedback loop: elevated wall shear stress during ejection activates endothelial nitric oxide synthase (eNOS) to produce NO \cite{kunert2015mechanobiological,ohhashi2023physiological}. This dynamic interplay establishes self-regulating pumping that enables the lymphatic system to maintain transport against adverse pressure gradients. A detailed account of the biomechanics of lymphatic flow control can be found in \cite{angeli2023biomechanical}. One of the goals of this paper is to understand this intrinsic pumping mechanism mathematically and identify the physiological parameter regimes leading to sustained (relaxation-)oscillations. 

The literature on mathematical modeling of the lymph flow is just emerging; however, there are several notable papers that set the overall approach (for a thorough review, see \cite{jayathungage2024computational,margaris2012modelling}). Models of lymph flow span from early lumped-parameter descriptions of single segments to recent image-based network-scale simulations, reflecting increasing physiological detail and computational complexity. Pioneering work by Reddy and Patel \cite{reddy1995mathematical} modeled flow through terminal lymphatics using mechanics-based relations between transmural pressure, vessel deformation, and valve resistance, establishing a 1D Poiseuille framework with compliant walls and nonlinear valve characteristics. Subsequent multi-lymphangion models represented collecting vessels as chains of actively contracting chambers separated by valves \cite{bertram2014development}. More recent contributions extend these ideas in several directions: reduced-order 1D models that incorporate electric pacemaking, NO-$Ca^{2+}$ feedback, and spatially distributed valves to study transport and wave propagation along vessels \cite{sedaghati20231d,kunert2015mechanobiological,contarino2018one,li2024fluid}, Darcy–Brinkman formulations for flow through porous lymph node \cite{giantesio2021model}, and multiscale CFD and network models \cite{girelli2024multiscale}.

We are interested in providing a mathematical model that describes the dynamical fluid behavior of the lymphatic fluid flowing through the lymphangion between systole and diastole. We couple the reduced Navier-Stokes equation with the chemical kinetics to derive governing equations that describe the dynamics of the area of the cross-section of the lymphangion (assumed axisymmetric), the fluid flux through it, and the concentrations of $Ca^{2+}$ and NO. The opening/closing of the valves is then triggered by the activation function that couples the mechanical and biochemical properties of the lymphangion. Although the lymph flow model has been present in the literature for many years \cite{reddy1995mathematical}, we derive it from first physical principles in a unified hemodynamical framework and couple it with the $Ca^{2+}$-NO kinetics. The main novel contributions of our research presented in this paper can be summarized as follows.
\begin{enumerate}
	\item A unified approach to systematic derivation of the lymph and blood flow model.
	\item Mathematical coupling of the intrinsic mechanics of lymphangion with the chemical kinetics of calcium ions and nitric oxide. 
	\item A detailed bifurcation analysis of the resulting non-smooth dynamical system yielding exact parameter regimes for which there exists a limit cycle.
\end{enumerate}
As the reader will see below, we provide a thorough analysis of the system of two non-smooth ODEs that represent the dynamics of $Ca^{2+}$ and NO. This system is analyzed under the simplifying assumption that changes in the radius of the lymphangion are negligible. Even under this restriction, we show that the model exhibits oscillatory behavior, verifying its explanatory potential of the physiological aspects of a lymphangion. The coupling between fluid mechanical and biochemical models leads to interesting nonlocal equations, which will be the subject of our subsequent work. Our model is robust and complex leaving a lot of space to investigate and conduct further research. 

This paper is structured as follows. In the next section, we derive the flow model starting from the Navier-Stokes equations, assuming radial symmetry, and using lubrication approximation. The model can describe the flow of both blood and the lymph however, we focus only on the latter. In Section 3 we discuss the biochemical oscillations driving the opening and closing of the lymphangion's valves. Since this is the most important mechanism that governs the behavior of the lymph flow at this level, we uncouple it from the fluid dynamical equation and show that such a simple system of two non-smooth ODEs can predict relaxation-oscillations. Our future work will focus on the analysis of the fully coupled model and its descriptive capabilities. 

\begin{table}
	\centering
	\begin{tabular}{cccc}
		\toprule  
		parameter & artery & vein & collecting lymphatic vessel \\ [0.5ex] 
		\midrule
        vessel radius $R_0$ & 1-15 mm \cite{muller2014global} & 0.8-8 mm \cite{muller2014global} & 0.05-1.1 mm \cite{contarino2018one} \cite{moore2018lymphatic}\\
        characteristic length $L$ & 200 mm \cite{muller2014global} & 300 mm \cite{muller2014global} & 3 mm \cite{jamalian2016network} \\
        flow velocity $V_x$ & 4.9–19 cm $s^{-1}$ \cite{klarhofer2001high} & 1.5–7.1 cm $s^{-1}$ \cite{klarhofer2001high} & 0.09-0.9 cm $s^{-1}$ \cite{zawieja2009contractile}\\
        pressure $P$ &   9.47-14.67 kPa \cite{woloszyn2012retrospective} &  0.67-1.33 kPa \cite{KAMATH2018233}&  0.5-0.8 kPa \cite{macdonald2008modeling}\\
        Reynolds number $Re$ & 1-4000 \cite{ku1997blood} & 13-360 \cite{saleem2023assessment}& 0.045-16 \cite{moore2018lymphatic} \\
        fluid density $\rho$ & 1060 kg $m^{-3}$ \cite{anliker1971nonlinear} & 1060 kg $m^{-3}$ \cite{anliker1971nonlinear} & 998 kg $m^{-3}$ \cite{macdonald2008modeling} \\
        fluid dynamic viscosity $\mu$ & 3.5-5.5 cP \cite{nader2019blood} & 3.5-5.5 cP \cite{nader2019blood} & 1 cP \cite{bertram2011chain}\\
		\\ [1ex] 
		\bottomrule
	\end{tabular}
    \label{tab:BloodLymph}
    \caption{Comparison between physical properties of blood (arteries, vein) and lymph (collecting lymphatic vessels). }
\end{table}

\section{Derivation of the flow model}
\subsection{Navier-Stokes equations and the constitutive law}
Although we will mainly be concerned with lymph flow, the derivation below is also equally valid for blood with a possible change in the rheology model. Quite a similar reasoning was given in \cite{smith2002anatomically}. The main assumptions we make are the following.
\begin{itemize}
	\item The flow is axisymmetric and laminar.
	\item The fluid is Newtonian. 
	\item The wall of the vessel responds dynamically to the flow with possible surface tension. 
	\item The flow is mainly driven by the pressure gradient. 
\end{itemize}
For reference, all quantities present in the model along with their numerical values are summarized in the Tab. \ref{tab:BloodLymph}. 

\begin{figure}
\centering
    \begin{tikzpicture}[scale=1.2, >=Stealth]
    
    \draw[->] (-0.2,0) -- (8.5,0) node[below right] {$x$ (axial)};
    \draw[->] (0,-0.2) -- (0,3.8) node[left] {$r$ (radial)};
    \node[below left] at (0,0) {centreline};
    
    \coordinate (A) at (0,0.9);
    \coordinate (B) at (1.8,1.4);
    \coordinate (C) at (3.6,2.6);
    \coordinate (D) at (5.2,1.9);
    \coordinate (E) at (7.4,1.1);
    
    \draw[very thick, domain=0:7.4, smooth, samples=200]
    plot coordinates {(0,0.9) (1.8,1.4) (3.6,2.6) (5.2,1.9) (7.4,1.1)};
    
    \draw[pattern=north east lines, pattern color=black!25] 
    plot coordinates {(0,0) (0,0.9) (0.8,1.05) (1.8,1.35) (3, 2.25) (3.6,2.6) (4.2,2.45) (5.2,1.85) (7.4,1.1) (7.4,0)}
    -- cycle;
    
    \coordinate (x0) at (3.2,0);                 
    \coordinate (Rpoint) at (3.2,2.38);         
    \draw[dashed] (x0) -- (Rpoint) node[midway, right] {};
    \draw (x0) node[below, yshift = -0.5] {$x_0$};
    \node[above, yshift = 3.5] at (Rpoint) {$\bigl(x_0,\,R(x_0,t)\bigr)$};
    \node[right] at ($(x0)!0.5!(Rpoint)$) {$R(x_0,t)$};
    
    \draw[->, thick] (0.8,0.6) -- ++(0.9,0.07) node[above, midway] {$u$};
    \draw[->, thick] (3.0,1.4) -- ++(1.0,0.05) node[above, midway] {$u$};
    \draw[->, thick] (5.0,0.9) -- ++(1.0,-0.02) node[below, midway] {$u$};
    
    \draw[->, thick, black] (3.45,1.95) -- ++(0.15,0.20) node[above right] {$v$};
    
    \draw[->, densely dashed] (6.4,2.6) arc (0:270:0.28) node[left=6pt] {$\theta$};
    
    \begin{scope}[shift={(9.2,2.5)}]
    \draw[thick] (0,0) circle (1.2) node[right=1.6cm] {cross-section at $x_0$};
    \draw[->] (1.2,0) arc (0:50:1.2) node[midway, right] {$\theta$};
    \draw[->, thick] (0,0) -- (1.2,0) node[midway, below, xshift = -1.2] {$R(x_0,t)$};
    \end{scope}
    
    \node[align=left, font=\small] at (4.8,3.5) {Dynamic boundary: $r = R(x,t)$};
    
    \end{tikzpicture}
    \caption{A diagrammatic cross-section of the tube. }
    \label{fig:CrossSection}
\end{figure}
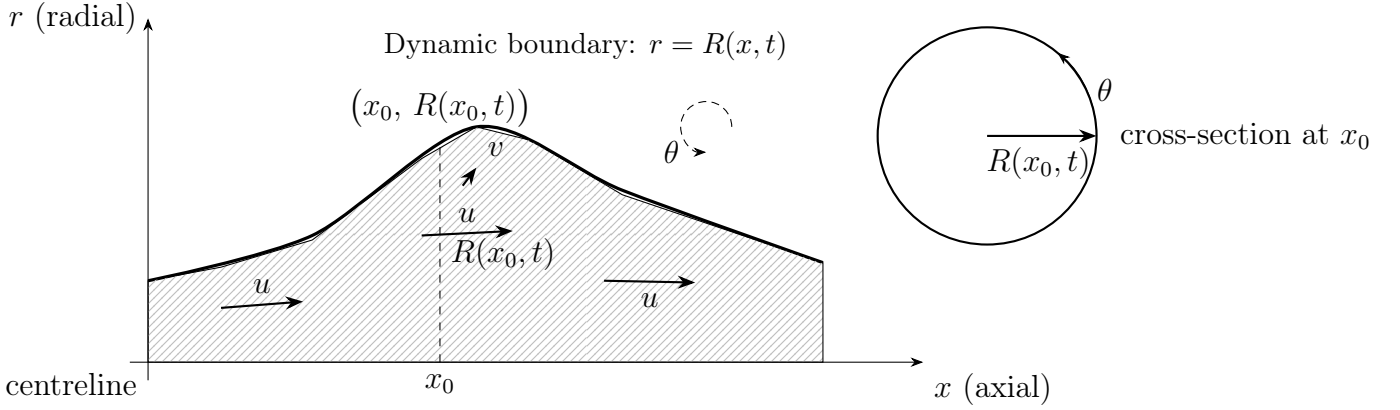

Consider a viscous flow through an axisymmetric tube with dynamic boundary (see Fig. \ref{fig:CrossSection}). Due to the setting, it is natural to write the Navier-Stokes equations in cylindrical coordinates $(r,\theta,x)$ and denote the profile of the vessel by $R=R(x,t)$ (that is, the radius of an instantaneous circle with an origin at $(0, 0, x)$). For both blood and lymph, the Reynolds number is usually small (see Tab. \ref{tab:BloodLymph}) and thus the flow is laminar. Due to symmetry, we also assume that the velocity is \emph{independent} of the angle $\theta$. If $u$ and $v$ denote the axial and radial components of the velocity, the conservation of momentum and mass give
\begin{equation}\label{eqn:NSDim}
	\begin{cases}
		u_t + u u_x + v u_r = - \dfrac{1}{\rho} p_x + \nu \left(u_{rr}+ \dfrac{1}{r} u_r + u_{xx}\right), \vspace{6pt}\\
		v_t + u v_x + v v_r = - \dfrac{1}{\rho} p_r + \nu \left(v_{rr}+\dfrac{1}{r} v_r - \dfrac{v}{r^2} + v_{xx}\right), \vspace{6pt}\\
		u_x + \frac{1}{r} \left(r v\right)_r = 0, \\
	\end{cases}
\end{equation}
where subscripts denote partial differentiation, $\rho$ is the density of the fluid, $p$ is the pressure, and $\nu$ is the kinematic viscosity. We are assuming that all biological fluids we consider are Newtonian, which can be a very good approximation to the lymph. However, blood is more complex due to the large relative size of red blood cells. Because of that, it is sometimes modeled by a non-Newtonian fluid of power-type rheology (for example, Carreau). However, in many situations, a constant viscosity model is adequate and sufficiently accurate. As mentioned above, the flow is axisymmetric and we subject it to the no-slip boundary condition, hence
\begin{equation}\label{eqn:Boundary}
	\begin{cases}
		u = 0, \quad v = 0, & \text{for } r = R, \\
		u_r = 0, & \text{for } r = 0. \\
	\end{cases}
\end{equation}
Moreover, since the fluid cannot penetrate the boundary $r = R(x,t)$ we have the \textit{kinematic boundary condition}
\begin{equation}\label{eqn:BoundaryKinametic}
	R_t + u R_x = v \quad \text{for} \quad r = R(x,t). 
\end{equation}

The above system \eqref{eqn:NSDim} of the equation must be provided by a constitutive law that describes the response of the transmural pressure $p$ to the flow (for some other models see \cite{contarino2018one})
\begin{equation}\label{eqn:PressureRadius}
	p - p_0 = G\; \phi\left(\frac{R}{R_0}\right) - T R R_{xx},
\end{equation}
where $G>0$ is the stretching pressure amplitude, $R_0$ is the reference value of the radius of the vessel (say, in a typical unstretched state) and $T$ is the tube tension. The function $\phi$ denotes the response of the vessel wall to the applied pressure. In the simplest, but still adequate and accurate, possible scenario, the precise form of this pressure-radius relation can be found empirically as a power law (usually stated in terms of the radius $A=\pi R^2$)
\begin{equation}\label{eqn:PressureRadiusPhi}
	\phi(z) = z^\gamma-1, \quad \gamma > 0.
\end{equation}
In particular, by the Laplace law, we would have $\gamma = 1$. For lymphangions, similarly to Rahbar et al. \cite{rahbar_pressure_diameter}, we can use the pressure-radius relation of the form
\begin{equation}\label{eqn:RahbarPressure-Radius}
    \phi(z) = P_0 \left(\exp(S_p(z-z_0)) + \alpha z^{-3} - \beta \right),
\end{equation}
The exponential term accounts for the rapid growth in the pressure values for a large radius, while $z^{-3}$ dominates for a smaller radius. Another approach is to use the reciprocal function instead of the exponent. In such a case, the pressure-radius relation is of the form:
\begin{equation}\label{eqn:ReciprocalPressure-Radius}
    \phi(z) = \frac{\gamma}{z_0 - \lambda z} - \alpha z^{-3} + \beta,
\end{equation}
where all the parameters can be found by fitting the least-squares to the real data. The comparison of data and fitted curves is presented in Figure \ref{fig:curvePressureRadius}.



\begin{figure}
    \centering
    \includegraphics[width=0.8\textwidth]{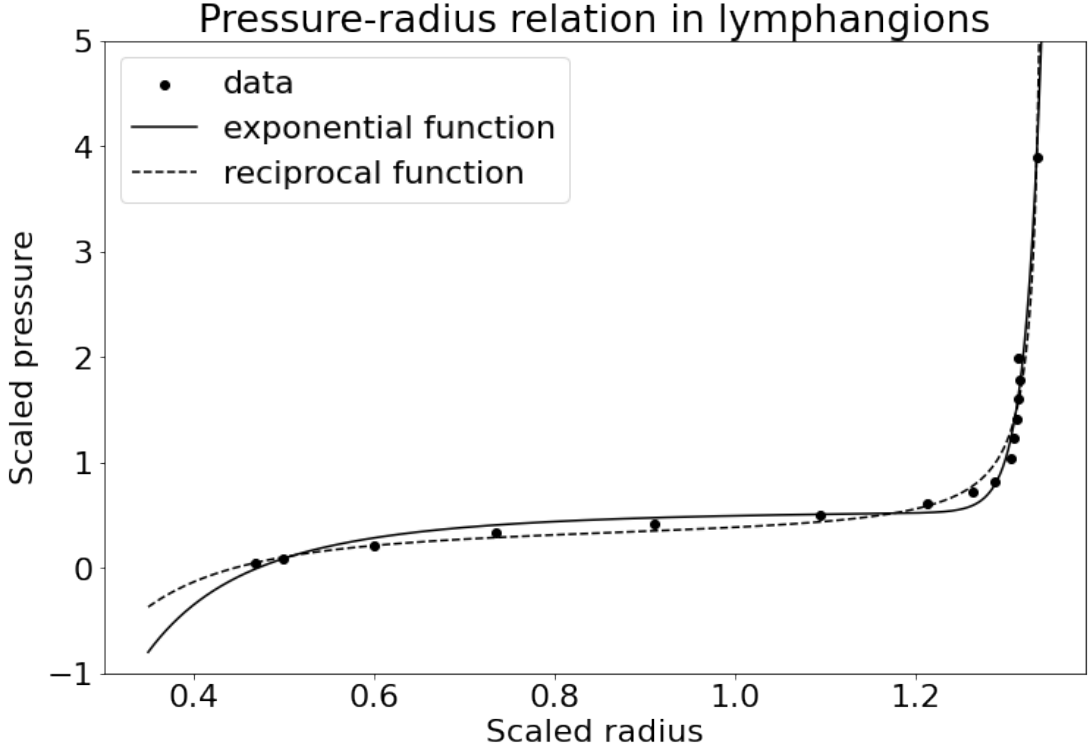}
    \caption{Curves \eqref{eqn:RahbarPressure-Radius}, \eqref{eqn:ReciprocalPressure-Radius} fitted to data describing pressure-radius relation.}
    \label{fig:curvePressureRadius}
\end{figure}

\subsection{Nondimensionalization and the lubrication approximation}
The next step in determining the flow model is to scale all the quantities appearing in \eqref{eqn:NSDim}. We choose the following
\begin{equation}\label{eqn:Scales}
	x \sim L, \quad r \sim R_0, \quad t \sim \frac{L}{V_x}, \quad u \sim V_x, \quad v \sim V_r, \quad p \sim P,
\end{equation}
where the typical values of the scales are given in Tab. \ref{tab:BloodLymph}. As can be seen, we have chosen the convective time scale $L/V_x$ (note that in pulsatile flows of angular frequency $\omega$ we would have chosen $t \sim 1/\omega$). Many experiments indicate that both for blood and lymphatic, the flow is essentially one-dimensional and the \textit{aspect ratio} is small, that is,
\begin{equation}\label{eqn:AspectRatio}
	\epsilon := \frac{R_0}{L} \ll 1.
\end{equation}
For example, for arteries $\epsilon \approx 0.04$, veins $\epsilon \approx 0.01 $, and lymphangions $\epsilon \approx 0.02$. In all of these cases, we can see that $\epsilon\ll 1$. Since the conservation of mass has to be retained in the same form from \eqref{eqn:NS} we have
\begin{equation}
	\frac{V_x}{L} u_x + \frac{V_r}{R_0} \frac{1}{r} \left(r v\right)_r = 0 \quad \implies \quad \frac{V_r}{V_x} = \frac{R_0}{L} = \epsilon,
\end{equation}
where, \textit{with the usual abuse of the notation}, we retained the same notation as before. Thus, the velocity components scale in the same way as the geometrical properties of the flow. 

We, thus, have a shallow flow which is governed by a \textit{nondimensional} system in which, 
\begin{equation}\label{eqn:NS}
	\begin{cases}
		u_t + u u_x + v u_r = - \dfrac{P}{\rho V_x^2} p_x + \dfrac{\nu L}{V_x R_0^2} \left(u_{rr}+ \dfrac{1}{r} u_r + \epsilon^2 u_{xx}\right), \vspace{6pt}\\
		\epsilon^2 \left(v_t + u v_x + v v_r \right) = - \dfrac{P}{\rho V_x^2} p_r + \epsilon^2\dfrac{\nu L}{V_x {R_0}^2} \left(v_{rr}+\dfrac{1}{r} v_r - \dfrac{v}{r^2} + \epsilon^2 v_{xx}\right), \vspace{6pt}\\
		u_x + \dfrac{1}{r} \left(r v\right)_r = 0. \\
	\end{cases}
\end{equation}
Now, in slender tubes: arteries, veins, and lymphangions, the flow is essentially driven by the pressure, so that we can determine the natural pressure scale 
\begin{equation}
	\dfrac{P}{\rho V_x^2} = \dfrac{\nu L}{V_x R_0^2} \quad \implies P = \frac{\rho V_x^2}{ \epsilon^2 Re} \quad \text{where the Reynolds number is} \quad Re := \frac{L V_x}{\nu}, 
\end{equation}
which essentially is the \textit{lubrication scaling}. Note that we could also have defined the Reynolds number as is done in pipe flow theory by $Re_r := R_0 V_x/\nu = \epsilon Re$. In this way, our equations become
\begin{equation}\label{eqn:NSScaled}
	\begin{cases}
		\epsilon^2 Re\left(u_t + u u_x + v u_r\right) = - p_x + \dfrac{1}{r}\left(r u_r\right)_r + \epsilon^2 u_{xx}, \vspace{6pt}\\
		\epsilon^4 Re \left(v_t + u v_x + v v_r \right) = - p_r + \epsilon^2 \left(v_{rr}+\dfrac{1}{r} v_r - \dfrac{v}{r^2} + \epsilon^2 v_{xx}\right), \vspace{6pt}\\
		u_x + \dfrac{1}{r} \left(r v\right)_r = 0. \\
	\end{cases}
\end{equation}
In addition to the dynamic equation, we also have to scale the constitutive one \eqref{eqn:PressureRadius}
\begin{equation}\label{eqn:PressureRadiusScaled}
	p - \epsilon^2 Re\frac{p_0}{\rho V_x^2} = \tau \phi(R) - \sigma RR_{xx},
\end{equation}
where the \textit{tube number}, the \textit{capillary number} $Ca$, and its scaled form $\sigma$ are defined by
\begin{equation}
	\tau := \epsilon^2 Re\frac{G}{\rho V_x^2}, \quad \sigma := \frac{\epsilon^4}{Ca}, \quad Ca := \frac{\rho \nu V_x}{T L}.
\end{equation}
Typical values of these quantities are as follows: $\tau = 32.7$, $\sigma=3.48 \times 10^{-3}$, $Ca=2.21 \times 10^{-5}$. The system \eqref{eqn:NSScaled} together with the constitutive equation \eqref{eqn:PressureRadiusScaled} is equivalent to the original flow equations \eqref{eqn:NS}. 

\subsection{The leading-order model}
Taking a leading-order approximation $\epsilon = 0$, we can very quickly obtain a closed model for the first approximation of the flow. For then, from \eqref{eqn:NSScaled} we have
\begin{equation}\label{eqn:NSScaledLeading}
	\begin{cases}
		p_x = \dfrac{1}{r}\left(r u_r\right)_r, \vspace{6pt}\\
		p_r = 0, \vspace{6pt}\\
		u_x + \dfrac{1}{r} \left(r v\right)_r = 0, \\
	\end{cases}
\end{equation}
with boundary conditions \eqref{eqn:Boundary}. These are typical lubrication equations with $r$-independent pressure and horizontal Poiseuille flow
\begin{equation}\label{eqn:Poiseuille}
	u(r,x,t) = -\frac{p_x(x,t)}{4} R^2 \left(1-\left(\frac{r}{R}\right)^2\right),
\end{equation}
where the pressure gradient is calculated from \eqref{eqn:PressureRadiusScaled} 
\begin{equation}
	p_x = \tau \phi'(R) R_x - \sigma (R_x R_{xx} + RR_{xxx}).
\end{equation}
Note that we have retained the tube tension term as in some situations it may have an influence on the flow \cite{macdonald2008modeling}. Now, we can integrate the conservation of mass equation in \eqref{eqn:NSScaledLeading} to obtain
\begin{equation}
	0 = \int_0^{R(x,t)} r u_x dr + R(x,t) v|_{r=R(x,t)} = \left(\int_0^{R(x,t)} r u dr\right)_x - R(x,t) R_x(x,t)u|_{r=R(x,t)}  + R(x,t) v|_{r=R(x,t)}.
\end{equation}
From the kinematic boundary condition \eqref{eqn:BoundaryKinametic} we can simplify to obtain
\begin{equation}
	R R_t + \left(\int_0^{R(x,t)} r u dr\right)_x = 0.
\end{equation}
Finally, since the velocity profile is known to be Poiseuille \eqref{eqn:Poiseuille} we can explicitly compute the integral and obtain the \textit{leading-order equation} for the radius of the vessel
\begin{equation}\label{eqn:leadingOrderRadius}
	R R_t = \frac{1}{16} \left(R^4 \left(\tau \phi'(R) R_x-\sigma R_x R_{xx} - \sigma RR_{xxx}\right)\right)_x,
\end{equation}
which is a nonlinear fourth-order dispersive diffusion equation. For boundary conditions, it is natural to prescribe both the value of the radius and the pressure gradient at each end of the vessel, that is,
\begin{equation}
    R(x,t) = R_{1,2}(t), \quad \tau \phi'(R(x,t)) R_x(x,t) - \sigma \Bigl(R(x,t) R_{xx}(x,t)\Bigr)_x  = (p_x)_{1,2}(t) \quad \text{for} \quad x = 0,1,
\end{equation}
which, together with a suitable initial condition, makes the problem well-posed. 

\subsection{Averaging the Navier-Stokes equations}
Note that the leading-order model assumes a steady flow of the fluid. To gain more insight into the transient features and for the completeness of our modeling, we have to go back to \eqref{eqn:NSScaled} and retain the left-hand side of the axial velocity but cancel the $\epsilon^2$ term assuming that $Re \approx O(10)$. From the radial velocity we still obtain the fact that the pressure does not change radially. Having that in mind, we proceed to averaging over the cross-sectional area. Alternatively, define the average of any axisymmetric quantity $f$ by
\begin{equation}\label{eqn:Average}
	\overline{f}(x,t) := \frac{1}{\pi R^2}\int_0^{2\pi}\int_0^R f(x,t,r) r dr d\theta = \frac{2}{R^2} \int_0^R f(x,t,r) r dr.   
\end{equation}
Next, multiply the conservation of mass equation in \eqref{eqn:NSScaled} by $r$ and integrate to, similarly as above, obtain
\begin{equation}
	0 = \int_0^R u_x rdr + R v|_{r=R} = \left(\int_0^R u rdr\right)_x - R \left(R_x u - v\right)|_{r=R}. 
\end{equation}
From the definition of the average \eqref{eqn:Average} and the kinematic boundary condition \eqref{eqn:BoundaryKinametic} we arrive at the \textit{ averaged conservation of the mass equation}
\begin{equation}\label{eqn:ConservationOfMassAvg}
	(R^2)_t + \left(\overline{u} R^2\right)_x = 0.
\end{equation}

Before we move to the momentum equation, we can use the conservation of mass to remove the $v$-related term in \eqref{eqn:NSScaled}, since
\begin{equation}
	(r v u)_r = (rv)_r u + r v u_r = - r u u_x + r v u_r \implies r v u_r = (r v u)_r + r u u_x
\end{equation}
Hence, multiplying by $r$ and integrating the first equation in \eqref{eqn:NSScaled} we can obtain the following
\begin{equation}
	\epsilon^2 Re\left(\left(\int_0^R u r dr\right)_t + \left(\int_0^R u^2 r dr\right)_x - R u|_{r = R}\left(R_t + u R_x - v\right)|_{r=R}\right) = - \frac{1}{2}R^2 p_x + R u_r|_{r=R},
\end{equation}
where we used the fact that $(u^2))x = 2 u u_x$. Recalling the definition of the cross-sectional average \eqref{eqn:Average} and the kinematic boundary condition \eqref{eqn:BoundaryKinametic} we finally obtain
\begin{equation}\label{eqn:NSAveraged0}
	\frac{1}{2}\epsilon^2 Re \left(\left(R^2 \overline{u}\right)_t + \left(R^2 \overline{u^2}\right)_x\right) = - \frac{1}{2}R^2 p_x + R u_r|_{r=R}.
\end{equation}
As usual, averaging the Navier-Stokes equation, we obtain a new variable $ \overline{u^2}$ that has to be related to the dynamical ones: $\overline{u}$ and $R$. The typical Bussinesq assumption yields
\begin{equation}
	\overline{u^2} = \alpha \overline{u}^2,
\end{equation}
with the shape factor $\alpha$ that has to be specified empirically. For example, the classical Poiseuille flow yields $\alpha = 8/3$ and in the experimentally confirmed power-law model \cite{smith2002anatomically}
\begin{equation}\label{eqn:PowerLawVelocity}
	u(r,x,t) = \frac{2+\gamma}{\gamma}\overline{u}(x,t) \left(1-\left(\frac{r}{R}\right)^\gamma\right), \quad \gamma > 0,
\end{equation} 
we have $\alpha = 2(2+\gamma)/(1+\gamma)$. The prefactor has been chosen to fix the cross-sectional average exactly to $\overline{u}$. The velocity profile in lymphangions is typically described as parabolic \cite{rahbar2011model, contarino2018one}. However, anatomical considerations suggest that this assumption may not be universally valid \cite{margaris2012modelling}. Because of that, we can assume that, in general, the velocity has the form
\begin{equation}\label{eqn:VelocityProfile}
	u(r,x,t) = \overline{u}(x,t) \psi\left(\frac{r}{R}\right) \quad \text{where} \quad \int_0^1 \psi(z) z dz = \frac{1}{2}, \quad \psi(R) = 0, \quad \psi'(0) = 0,
\end{equation}
where the normalization condition on $\psi$ forces the cross-sectional average of $u$ to be equal to $\overline{u}$. Note that by the above, the boundary conditions \eqref{eqn:Boundary} are automatically satisfied. In this general case, the shape coefficient is equal to
\begin{equation}
	\alpha = 4\int_0^1 \psi(z)^2 z dz,
\end{equation}
and can be considered as a known value. Therefore, assuming the experimentally confirmed profile \eqref{eqn:VelocityProfile} in \eqref{eqn:NSAveraged0} we obtain the \textit{averaged momentum equation}
\begin{equation}
	\epsilon^2 Re \left(\left(R^2 \overline{u}\right)_t + \alpha\left(R^2 \overline{u}^2\right)_x\right) = - R^2 \left(\tau \phi'(R) R_x - \sigma R_x R_{xx} - \sigma R R_{xxx}\right) + 2\psi'(1) \overline{u},
\end{equation}
which, together with the conservation of mas \eqref{eqn:ConservationOfMassAvg}, form a closed system of two nonlinear PDEs for the mean velocity $\overline{u}=\overline{u}(x,t)$ and the radius of the vessel $R=R(x,t)$. We can now revert the scaling and put the governing equations in the dimensional form
\begin{equation}\label{eqn:MainSystemDimensional}
	\begin{cases}
		(R^2)_t + (U R^2)_x = 0, \\
		\left(U R^2\right)_t + \alpha\left(U^2 R^2\right)_x = -\dfrac{R^2}{\rho} \left(G\phi'\left(\dfrac{R}{R_0}\right) R_x- T R_x R_{xx} - T R R_{xxx}\right) + 2\nu \,\psi'(1) U, \\
	\end{cases}
\end{equation}
where, we denote the velocity averaged by dimensional area by $U=U(x,t)$ and retained the same notation for $R$ as in the nondimensional case. Note that the above can be neatly simplified when we consider the area $A=\pi R^2$ and the convective flux $Q= U A$ as dynamical variables, that is, we can also consider
\begin{equation}\label{eqn:MainSystemDimensional}
	\begin{cases}
		A_t + Q_x = 0, \\
		Q_t + \alpha\left(U Q\right)_x = -\dfrac{1}{\rho}A \; p_x + 2\nu \,\psi'(1) \dfrac{Q}{A},
	\end{cases}
\end{equation}
Notice that the left-hand sides of these are also present in the shallow-water equations (in the Saint-Venant model) This equation has previously been derived in the context of vascular flow by several authors. Usually by an ad hoc procedure \cite{reddy1995mathematical} and sometimes by, similar to ours, a rigorous asymptotic analysis of the Navier-Stokes equations \cite{smith2002anatomically}. Here, thanks to the scaling we have been able to justify when the above system is valid for the description of the blood and lymph flow. Note also that we can use only one model to compute the flow through the vessel regardless of whether it is an artery, vein, or lymphangion. To distinguish between them, we only have to specify the appropriate compliance function $\phi$, velocity profile $\psi$, and decide whether it is meaningful to take into account the vessel tension $T$ that can be meaningful for the lymph.

\subsection{Valve boundary conditions for the lymphangion}
The above derived dynamic equations are valid for all physiological flows considered: blood and lymph. As explained in the Introduction, the latter is much less understood and we devote to it the remainder of this paper. 

We can characterize the lymphangion's valve state as open or closed. To idealize, assume that for these two states we have $R = R_-$ for the closed valve and $R = R_+$ for the open valve. The triggering between these two states occurs when some quantity crosses through its threshold value. For example, in \cite{bertram2011chain} the authors considered the opening of the value if the pressure gradient was large enough. This suggests that the boundary condition can be of the form
\begin{equation}\label{eqn:BoundaryConditionP}
    R = R_- + (R_+ - R_-)f(\Delta p - P) \quad \text{at the boundary} \quad x = 0,1,
\end{equation}
where $f$ is the activation function, for example, discontinuous Heaviside or a smooth sigmoid. That is, we can model the opening and closing of the valve using a switching function triggered by an external mechanism \cite{contarino2018one}.

In what follows, we analyze a different mechanism for triggering the valve state change. The main premise for this mechanism is the biochemical oscillations of nitric oxide (NO) and calcium ions (described in the following in detail). The presence of nitric oxide relaxes the lympgangion's wall by pumping the fluid through an open valve. As the lymphangion is filled, the wall shear stress decreases, inhibiting NO production. Then, through several channels, calcium ions are produced, leading to a contraction of the lymphangion that results in the opening of the downstream valve. The flow increases the wall shear stress and, hence, NO starts to be produced again. The process continues in this oscillatory fashion. We will describe this period behavior in detail in the next section, but now we can claim that the boundary value of $R$ can also depend on the concentration of calcium ions $C$, that is,
\begin{equation}
    R = R_- + (R_+ - R_-)(1-f(C - C_0)) \quad \text{at the boundary} \quad x = 0,1, 
\end{equation}
where $C_0$ is the threshold value required to contract the lymphangion. The important point is that, as we noted, concentrations of nitric oxide and calcium ions are closely intertwined. The concentration of the former is driven by the wall sheer stress $\sigma$, which for the Poisseulie flow is given by (for other profiles a similar relation holds)
\begin{equation}\label{eqn:BoundaryConditionS}
    \sigma = - \frac{R\Delta p}{2L},
\end{equation}
where $L$ is the length of the lymphangion. We therefore see that in reality both \eqref{eqn:BoundaryConditionP} and \eqref{eqn:BoundaryConditionS} are in some sense equivalent when it comes to modeling. The common link is shear stress that connects the concentration of the chemical species with the radius of the lymphangion and the pressure gradient. In our analysis, we choose to describe the biochemical mechanisms because, in our opinion, it is much more fundamental. 

\section{Biochemical dynamics model}
Having derived two models that describe lymphangion fluid mechanics, we now proceed to the main result of our paper - analysis of the chemical kinetics required for the mechanism of valve opening and closing. Since it is not our aim here to analyze the coupled system: fluid + biochemistry, from now on we will assume that the radius of the lymphangion $R$ is constant, for example, we can take its mean value throughout its length. Even with this assumption, we will show that the biochemical model still exhibits periodic oscillations. It is an objective of our future work to analyze the completely coupled model and investigate how the changes in the radius affect parameters of the chemical oscillations. 

\subsection{Derivation}
Lymphatic pumping is regulated by biochemical processes. Among the key regulators of this process are calcium ions (Ca$^{2+}$) and nitric oxide, whose interaction governs the contractility and relaxation cycles of lymphatic vessels. Calcium ions (Ca$^{2+}$) are critical regulators of lymphatic vessel contractions. Similarly as in blood vessels, Ca$^{2+}$ influx initiates the contraction of lymphatic muscle cells.  

In our model, we consider the following aspects of dynamics of calcium ions: 
\begin{itemize}
    \item First order decay, enhanced by NO concentration, which can modulate the activity of the calcium clearance mechanism.
    \item Voltage-dependent calcium channels, that split to L-type (“longlasting”) and T-type (“transient") channels \cite{munn2015mechanobiology}.
    \item Ca$^{2+}$ influx from stretch-activated ion channels. The vessel responds to an increase in luminal pressure by constricting \cite{munn2015mechanobiology}, \cite{kunert2015mechanobiological}.
    \item Ca-dependent calcium channels, activated when Ca$^{2+}$ concentration exceeds the threshold level \cite{munn2015mechanobiology}. 
\end{itemize}
The descriptions, meaning and typical numerical values of all parameters used are presented in Tab. \ref{tab:Parameters}. The final equation is of the form:
\begin{equation}\label{eq:dCa_dt}
    \frac{dC}{dt} =
        -k_{Ca}^{-}(1+mN)C + k_{Ca}^{+} + k_{Ca}^{+} S\left(\frac{R}{R_{crit}}\right) + 10k_{Ca}^{+} \mathcal{H}(C - C_{tresh})
\end{equation}
\normalcolor
where $\mathcal{H}$ is the Heaviside function. 

NO (nitric oxide) is a gas synthesized from L-arginine by NO synthase (NOS) in vascular endothelial cells \cite{ohhashi2023physiological}. We distinguish three types of NOS \cite{govers2001cellular}:
 \begin{itemize}
     \item eNOS,
     \item neuronal NOS (nNOS),
     \item cytokine-inducible NOS (iNOS).
 \end{itemize}
The first two types can be activated rapidly by an increase in Ca$^{2+}$, leading first to activation by phosphorylation of NOS and finally to the subsequent release of NO. On the other hand, increased lymph flow generates shear stress, which stimulates NO production through NOS activation \cite{munn2015mechanobiology}. 
We describe the NO dynamics with:
\begin{itemize}
    \item exponential decay \cite{munn2015mechanobiology},
    \item production proportional to shear stress $\tau_{xx}$ \cite{munn2015mechanobiology}.
\end{itemize}
We relate the activation of the shear stress-based mechanism with the level $Ca^{2+}$
\begin{equation}\label{eq:dNO_dt}
    \frac{dN}{dt} = -k_{NO}^{-}N + \mathcal{H}(C - C_{shear}) k_{NO}^{+} \left( \frac{\tau_{xx}}{\tau_{ref}}\right).
\end{equation}
To facilitate further analysis, we express \eqref{eq:dCa_dt} and \eqref{eq:dNO_dt} in nondimensional form. A natural choice for the concentration of Ca$^{2+}$ is, of course, $C_{shear}$ as at this value the discontinuity of the flow appears. Moreover, from \eqref{eq:dNO_dt} we see that the equation would be simplified provided that we scale $N$ with $m^{-1}$. Finally, for the time scale we choose $k_{NO}^{-}$, that is, the scales are the following
\begin{equation}
    C \sim C_{shear}, \quad N \sim \frac{1}{m}, \quad t \sim \frac{1}{k_{NO}^{-}}.
\end{equation}
As usual, to avoid cluttering the notation, we retain the original names for the dependent and independent variables. From now on, we will work only in the nondimentional form. Note that, after the scaling, the jump in the first equation becomes simply $\mathcal{H}(C-1)$. By elementary computations, we obtain the scaled system
\begin{equation}\label{eq:alternative_scaling_system}
    \begin{cases}
        \displaystyle{C' = -\alpha (1+N) C + \beta + \gamma \mathcal{H}(C-1), }\\
        \displaystyle{N' = -N + \zeta \mathcal{H}\left(C-\frac{C_{tresh}}{C_{shear}}\right),}
    \end{cases}
\end{equation}
\normalcolor
where the nondimensional parameters are defined by
\begin{equation}
    \alpha := \frac{k_{Ca}^-}{k_{NO}^{-}}, \quad \beta=\frac{k_{Ca}^+\left( 1 + S\left(\frac{R}{R_{crit}}\right)\right)}{k_{NO}^{-} C_{shear}}, \quad
    \gamma=\frac{10 k_{Ca}^+}{k_{NO}^{-} C_{shear}}, \quad
    \zeta = \frac{k_{NO}^+ m \left(\frac{\tau_{xx}}{\tau_{ref}}\right)}{ k_{NO}^{-}},
\end{equation}
where $\tau_{xx}=\frac{\Delta p R}{2 l}$.
The values of the parameters are summarized in Table \ref{tab:Parameters}. The typical values of nondimensional parameters are presented in Table \ref{tab:NondimensionalParameters} and, as we can see, all are of order of unity. As $S\left( \frac{R}{R_{crit}}\right)$ we take $\left( \frac{R}{R_{crit}}\right)^{11}$ as suggested in \cite{kunert2015mechanobiological}, while as $\tau_{xx}$ we use the typical value of 1 Pa \cite{angeli2023biomechanical}.
\begin{table}
\centering
\begin{tabular}{cccc}
  \toprule
  symbol & description & value & reference \\
  \midrule
  $k_{NO}^{-}$ & NO degradation rate constant & $75.1 s^{-1}$ & \cite{li2024fluid} \\ 
  $k_{NO}^{+}$ & NO production rate constant  & 20 & \cite{li2024fluid} \\ 
  $k_{Ca}^{-}$ & $Ca^{2+}$ degradation rate constant & $375.9 s^{-1}$ & \cite{li2024fluid} \\
  $k_{Ca}^{+}$ & $Ca^{2+}$ production rate constant & $2.5 s^{-1}$ & estimated \\
  $m$          & Rate constant for NO inhibition of $Ca^{2+}$ & 0.5 & $ \cite{li2024fluid}$ \\
  $R$ & Vessel radius & $R_0$& \\
  $R_{crit}$   & Value fo $R$ activating stretch-ativated channels & $0.77 R_0$ & estimated \\
  $\tau_{ref}$ & Shear stress activating NO production & 0.1 Pa & \cite{angeli2023biomechanical} \\ 
  $C_{shear}$  & $Ca^{2+}$ concentration activating SS-based mechanism & 0.1 &  \\
  $C_{tresh}$  & $Ca^{2+}$ concentration activating calcium channels & 0.1 & \cite{li2024fluid} \\
  \bottomrule
\end{tabular}
\caption{Physical quantities of the model.}
\label{tab:Parameters}
\end{table}

\begin{table}
\centering
\begin{tabular}{cc}
  \toprule
  parameter & value \\
  \midrule
  $\alpha$ & 5.01 \\
  $\beta$ & 6.23 \\
  $\gamma$ & 3.33 \\
  $\zeta$ & 1.07 \\
  \bottomrule
\end{tabular}
\caption{Nondimensional parameters of the model.}

\label{tab:NondimensionalParameters}
\end{table}

\subsection{Analysis of the nonsmooth system}
In this section, we analyze the simplified system \eqref{eq:alternative_scaling_system}, with $R=const.$ being the typical value of radius and $C_{tresh}=C_{shear}$. We show that despite its reduced complexity, the system exhibits nontrivial dynamical behavior. In particular, for an appropriate choice of parameters, the system admits a periodic limit cycle.

Let 
\begin{equation}
    \mathbb{R}_{+}^{2}= \{x=(N, C)^T | N \geq 0,  C \geq 0\},
\end{equation}
and
\begin{equation}
    H(x, \mu) := C - 1,
\end{equation}
be the smooth scalar function with non-zero gradient. With $\mu \in \mathbb{R}$ we indicate the dependence on the parameter. We omit this notation where necessary. Then we can define 
\begin{equation}\label{eq:sigmas}
    \Sigma^{-} = \{x \in \mathbb{R}_{+}^{2} \quad | \quad H(x, \mu) < 0 \}, \quad \Sigma^{+} = \{x \in \mathbb{R}_{+}^{2} \quad | \quad H(x, \mu) > 0 \},
\end{equation}
which are two smooth vector fields separated by the boundary:
\begin{equation}
    \Sigma = \{x \in \mathbb{R}_{+}^{2} \quad | \quad H(x) = 0 \}.
\end{equation}
The system \eqref{eq:alternative_scaling_system} can thus be rewritten as:
\begin{equation}
    Z'(t) = \begin{cases}
        F_1(x, \mu), & x \in \Sigma^{-}\\
        F_2(x, \mu), & x \in \Sigma^{+}
    \end{cases},
\end{equation}
where we define
\begin{equation}
    F_1(x, \mu) := (-N, \, -\alpha (N+1)C + \beta)^T, \quad F_2(x, \mu) := (-N + \zeta, \, -\alpha (N+1)C + \beta + \gamma)^T.
\end{equation}
To classify the boundary, we need to check the sign of the Lie derivative. We denote the Lie derivative as $F.H(x) = \langle \nabla F(x), H(x) \rangle$ and the i-th Lie derivative as $F^i.H(x) = \langle \nabla F^{i-1}(x), H(x) \rangle$. We distinguish the following regions on the discontinuity set $\Sigma$:
\begin{enumerate}
    \item $\Sigma_C \subseteq \Sigma$ is the \emph{crossing region} if $(F_2.H(x))(F_1.H(x)) > 0$ in $\Sigma_C$
    \item $\Sigma_s \subseteq \Sigma$ is the \emph{attracting sliding region} if $(F_2.H(x)) < 0$ and $(F_1.H(x)) > 0$ on $\Sigma_s$
    \item $\Sigma_e \subseteq \Sigma$ is the \emph{repelling sliding region} if $(F_2.H(x)) > 0$ and $(F_1.H(x)) < 0$ on $\Sigma_e$
\end{enumerate}
The point $x \in \Sigma$, such that $(F_2.H(x))(F_1.H(x)) = 0$ is called a \textbf{tangent point} or \textbf{tangential singularity} \cite{buzzi2018poincare} \cite{bernardo2008piecewise}. It is a tangent contact point between the trajectories of $F_1$ and/or $F_2$ with $\Sigma$.
In the case of our system, after simple calculations, we obtain:
\begin{equation}
    F_1.H(x) = -\alpha(N+1)C + \beta, \quad  F_2.H(x) = -\alpha(N+1)C + \beta + \gamma.
\end{equation}
As the product of $F_2.H(x)$ and $F_1.H(x)$ is rather complex, we provide Figure \ref{img:lie_derivatives} and Figure \ref{fig:three_boundary_cases} as a visual aid. The solid and dashed lines represent the points where $F_2.H(x)$ and $F_1.H(x)$, respectively, are equal to zero. The formulas are:
\begin{itemize}
    \item for $F_2.H(x)$:
    \begin{equation}
        C(N) = \frac{\beta}{\alpha(N+1)}
    \end{equation}
    \item for $F_1.H(x)$:
    \begin{equation}
        C(N) = \frac{\beta + \gamma}{\alpha(N+1)}
    \end{equation}
\end{itemize}

\begin{figure}[h!]
    \centering
    \begin{tikzpicture}
    \begin{axis}[
        axis lines=left,
        xlabel={$N$},
        ylabel={$C(N)$},
        domain=0:10,
        samples=200,
        ymin=0,
        ymax=5.5,
        xtick=\empty,
        ytick={1.5,4.5},
        yticklabels={$ \frac{\beta}{\alpha} $, $ \frac{\beta+\gamma}{\alpha} $},
        width=0.8\textwidth,
        height=0.5\textwidth,
        every axis plot/.append style={black},
        axis line style={black},
        tick style={black},
        every axis label/.append style={black},
        every tick label/.append style={black},
        every axis y label/.style={at={(axis description cs:-0.08,0.5)},anchor=south,rotate=90},
        every axis x label/.style={at={(axis description cs:0.5,-0.08)},anchor=north},
        axis on top,
        clip=false
    ]

    \addplot[thick,solid,domain=0:10] { (4.5)/(x+1) };

    \addplot[thick,dashed,domain=0:10] { (1.5)/(x+1) };

    \addplot[densely dotted] coordinates {(0,2.5) (10,2.5)};
    \addplot[densely dotted] coordinates {(0,1.5) (10,1.5)};


    \node[align=center, font=\tiny, text=black] at (axis cs:2,2.9)
    {$F_1.H(X)<0$};
    \node[align=center, font=\tiny, text=black] at (axis cs:2,2.6)
    {$F_2.H(X)<0$};
    
    \node[align=center, font=\tiny, text=black] at (axis cs:0.7,1.6)
    {$F_1.H(X)<0$};
    \node[align=center, font=\tiny, text=black] at (axis cs:1.1,1.3)
    {$F_2.H(X)>0$};
    
    \node[align=center, font=\tiny, text=black] at (axis cs:1.6,0.2)
    {$F_1.H(X)>0,\; F_2.H(X)>0$};

    \end{axis}
    \end{tikzpicture}
    \caption{Example realization of $F_{\Sigma^+}.H(x)$ (solid line) and $F_{\Sigma^-}.H(x)$ (dashed line).}
    \label{img:lie_derivatives}
\end{figure}
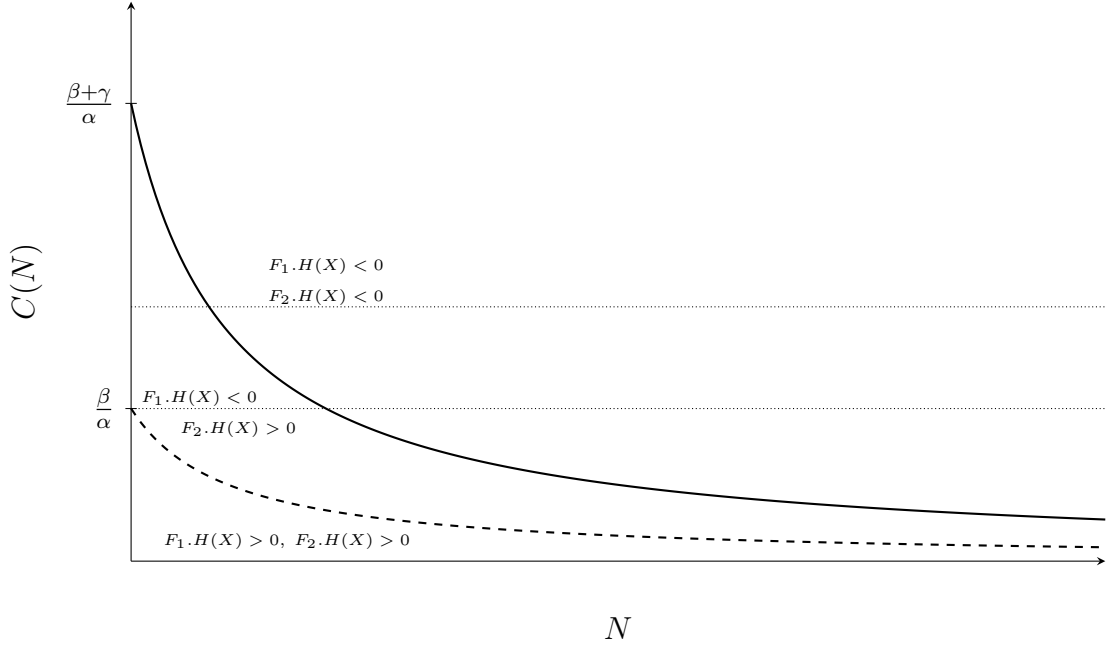

Depending on the values of parameters $\alpha, \beta, \gamma$, we may encounter one of the following cases:
\begin{itemize}
    \item $\frac{\beta}{\alpha} < \frac{\beta + \gamma}{\alpha} \leq 1$ - the entire boundary is classified as crossing region. If $\frac{\beta + \gamma}{\alpha}=1$, at $(0, 1)$ there exists a tangential singularity and the rest of the boundary is classified as the crossing region. 
    \item $\frac{\beta}{\alpha} \leq 1 < \frac{\beta + \gamma}{\alpha}$ - boundary consists of repelling sliding region, tangential singularity and crossing region. If $\frac{\beta}{\alpha} =1$, the boundary conists of tangential singularity, escaping sliding region, another tangential singularity and crossing region. 
    \item $1 < \frac{\beta}{\alpha} < \frac{\beta + \gamma}{\alpha}$ - boundary consists of crossing region, tangential singularity, repelling sliding region, tangential singularity and crossing region.
\end{itemize}

\begin{figure}[h!]\label{fig:three_boundary_cases}
    \centering
    \begin{tikzpicture}

    \pgfplotsset{
    commonaxis/.style={
        axis lines=left,
        xlabel={$N$},
        ylabel={$C(N)$},
        domain=0:10,
        samples=200,
        ymin=0,
        ymax=1.6,
        xtick=\empty,
        xticklabels={},
        xtick style={draw=none},
        yticklabel style={anchor=east, xshift=-2pt},
        width=0.3\textwidth,
        height=0.35\textwidth,
        every axis plot/.append style={black},
        axis line style={black},
        tick style={black},
        every axis label/.append style={black},
        every tick label/.append style={black},
        every axis title/.append style={black},
        every axis y label/.style={
            at={(axis description cs:-0.25,0.5)},
            anchor=south,
            rotate=90
        },
        every axis x label/.style={
            at={(axis description cs:0.5,-0.08)},
            anchor=north
        },
        axis on top,
        clip=false
    }
}
    \begin{scope}
        \begin{axis}[
            commonaxis,
            title={\textcolor{black}{$\frac{\beta}{\alpha} < \frac{\beta + \gamma}{\alpha} \leq 1$}},
            ytick={0.5,0.8},
            yticklabels={$ \tfrac{\beta}{\alpha} $, $ \tfrac{\beta+\gamma}{\alpha} $}]
            \addplot[thick,solid] {0.8/(x+1)};
            \addplot[thick,dashed] {0.5/(x+1)};
            \addplot[ultra thick,dash dot] coordinates {(0,1) (10,1)};
            \node[anchor=south, yshift=2pt, text=black] at (axis cs:10,1) {$C=1$};
        \end{axis}
    \end{scope}

    \begin{scope}[xshift=6cm]
        \begin{axis}[
            commonaxis,
            title={\textcolor{black}{$\frac{\beta}{\alpha} \leq 1 <\frac{\beta + \gamma}{\alpha}$}},
            ytick={0.8,1.2},
            yticklabels={$ \tfrac{\beta}{\alpha} $, $ \tfrac{\beta+\gamma}{\alpha} $}
        ]
            \addplot[thick,solid] {1.2/(x+1)};
            \addplot[thick,dashed] {0.8/(x+1)};
            \addplot[ultra thick,dash dot] coordinates {(0,1) (10,1)};
            \node[anchor=south, yshift=2pt, text=black] at (axis cs:10,1) {$C=1$};
        \end{axis}
    \end{scope}

    \begin{scope}[xshift=12cm]
        \begin{axis}[
            commonaxis,
            title={\textcolor{black}{$1 < \frac{\beta}{\alpha} < \frac{\beta + \gamma}{\alpha}$}},
            ytick={1.2,1.4},
            yticklabels={$ \tfrac{\beta}{\alpha} $, $ \tfrac{\beta+\gamma}{\alpha} $}
        ]
            \addplot[thick,solid] {1.4/(x+1)};
            \addplot[thick,dashed] {1.2/(x+1)};
            \addplot[ultra thick,dash dot] coordinates {(0,1) (10,1)};
            \node[anchor=south, yshift=2pt, text=black] at (axis cs:10,1) {$C=1$};
        \end{axis}
    \end{scope}

    \end{tikzpicture}
    \caption{Three main cases of boundary classification. }
\end{figure}
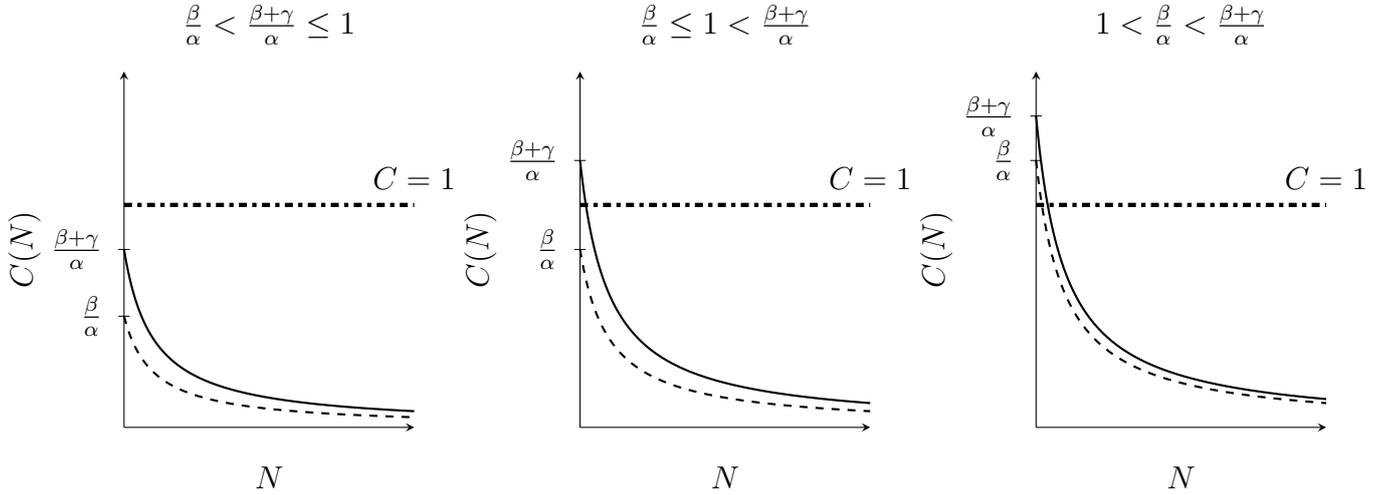
The phase planes corresponding to three cases of boundary classification described above are depicted Figure \ref{fig:boundary_phase_plane}.
\begin{figure}
    \centering
    \includegraphics[width=\linewidth]{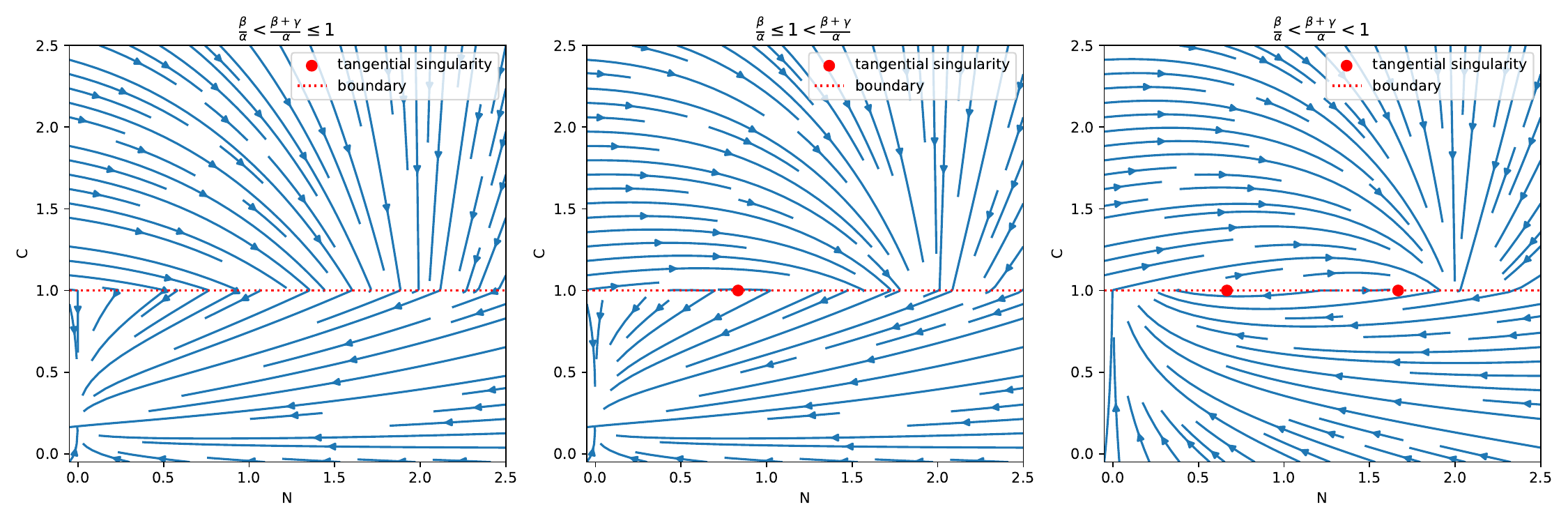}
    \caption{Three main cases of boundary classification - phase planes.}
    \label{fig:boundary_phase_plane}
\end{figure}
In our system, we have two tangent points: $S_1 = \left(\frac{\beta}{\alpha} -1, 1 \right)$ and $S_2 = \left(\frac{\beta + \gamma}{\alpha} -1, 1 \right)$. To classify them, we need to check the signs of the higher Lie derivatives. The second Lie derivative for $F_1$ is:
\begin{equation}\label{eqn:second_Lie_minus}
    F_1^{2}.H(x) = \alpha^2N^2C + 2\alpha^2N C + \alpha N C + \alpha^2 C - \alpha B N - \alpha \beta,
\end{equation}
while for $F_2$:
\begin{equation}\label{eqn:second_Lie_plus}
    F_2^{2}.H(x) = \alpha^2 N^2 C + 2 \alpha^2 N C + \alpha N C + \alpha^2 C - \alpha \zeta C - \alpha \beta N - \alpha \gamma N - \alpha \beta - \alpha \gamma.
\end{equation}
The value of $F_1^{2}.H(x)$ at the point $S_1$ is $\beta - \alpha$ and for the point $S_2$ the value is $\beta^2 + \beta \gamma + \gamma + \beta - \alpha$. The value of $F_2^{2}.H(x)$ at the point $S_1$ is $\beta - \alpha - \alpha \zeta - \beta \gamma$ and for the point $S_2$ the value is $\beta + \gamma - \alpha(\zeta + 1)$.
\begin{definition}\cite{bernardo2008piecewise}
    We say that the point $x$ is \textbf{admissible} equilibrium if:
    \begin{equation}
        F_1(x, \mu) = 0, \quad H(x, \mu) < 0.
    \end{equation}
or 
    \begin{equation}
        F_2(x, \mu) = 0, \quad H(x, \mu) > 0.
    \end{equation}
We say that the point $x$ is \textbf{virtual} equilibrium if:
    \begin{equation}
        F_1(x, \mu) = 0, \quad H(x, \mu) > 0.
    \end{equation}
or 
    \begin{equation}
        F_2(x, \mu) = 0, \quad H(x, \mu) < 0.
    \end{equation}
\end{definition}
\begin{definition}\cite{bernardo2008piecewise}
    We call a point $\tilde{x}$ a \textbf{pseudo-equilibrium} if it is an equilibrium of the sliding flow, i.e. for some scalar $\alpha_P$,
    \begin{equation}
        F_1(\tilde{x}, \mu) + \alpha_P (F_2 - F_1) = 0, \quad H(\tilde{x}, \mu) = 0.
    \end{equation}
    We call a pseudo-equilibrium \textbf{admissible} if $0 < \alpha_P < 1.$
Alternatively, we say that a pseudo-equilibrium is \textbf{virtual} if
$\alpha_P < 0$ or $\alpha_P > 1$.
\end{definition}
\begin{definition}\cite{bernardo2008piecewise}A point $\hat{x}$ is termed a \textbf{boundary equilibrium} of if
\begin{equation}
    F_1(\hat{x}, \mu)=0 \text{ or } F_2(\hat{x}, \mu)=0, \quad H(\hat{x}, \mu)=0.
\end{equation}   
\end{definition}
\noindent For $F_1$ equilibrium is given with the formula:
\begin{equation}\label{eqn:critical_point_minus}
    P_{st}^{-} = \left(0, \frac{\beta}{\alpha}\right).
\end{equation}
Note that this is the:
\begin{itemize}
    \item admissible equilibrium if $C<1$ that is, $\beta < \alpha$;
    \item boundary equilibrium if $C=1$, that is, $\beta=\alpha$;
    \item virtual equilibrium if $C>1$, that is $\beta > \alpha.$
\end{itemize}
The formula for Jacobi matrix is:
\begin{equation}
    F_{1,x} = \begin{pmatrix}
        -1 & 0 \\
        -\alpha C & -\alpha N - \alpha
        \end{pmatrix}.
\end{equation}
At point $P_{st}^{-}$ we have:
\begin{equation}
    F_{1, x} - \lambda I = \begin{pmatrix}
        -1-\lambda & 0 \\
        -\beta & -\alpha - \lambda
        \end{pmatrix}.
\end{equation}
and eigenvalues are $\lambda_{1, 1} = -1$, $\lambda_{2, 1} = -\alpha < 0.$ For $\alpha \neq 1$ point $P_{st}^{-}$ is a \textbf{stable node}. For $\alpha=1$ we have a \textbf{stable degenerate node}.
For $F_2$ the equilibrium is given with the formula:
\begin{equation}\label{eqn:critical_point_plus}
    P_{st}^{+} = \left(\zeta, \frac{\beta + \gamma}{\alpha (\zeta + 1)} \right).
\end{equation}
This is the:
\begin{itemize}
    \item admissible equilibrium if $C>1$ that is, $\beta > \alpha(\zeta + 1) - \gamma.$;
    \item boundary equilibrium if $C=1$, that is, $\beta = \alpha(\zeta + 1) - \gamma$;
    \item virtual equilibrium if $C<1$, that is $\beta < \alpha(\zeta + 1) - \gamma.$
\end{itemize}
The formula for Jacobi matrix is:
\begin{equation}
    F_{2,x} = \begin{pmatrix}
        -1 & 0 \\
        -\alpha C & -\alpha N - \alpha
        \end{pmatrix}.
\end{equation}
At point $P_{st}^{+}$ we have:
\begin{equation}
     F_{2, x} - \lambda I = \begin{pmatrix}
        -1-\lambda & 0 \\
        -\frac{\beta + \gamma}{\zeta + 1} & -\alpha (\zeta + 1)- \lambda
        \end{pmatrix}.
\end{equation}
and eigenvalues are $\lambda_{1, 2} = -1$, $\lambda_{2, 2} = -\alpha (\zeta + 1) < 0.$ For $\alpha \neq \zeta + 1$ point $P_{st}^{+}$ is a \textbf{stable node}. For $\alpha = \zeta + 1$ point $P_{st}^{+}$ is a \textbf{stable degenerate node}.
\newline The pseudo-equilibrium of the system is:
\begin{equation}
    P_{ps} = \left(\zeta \frac{\alpha  - \beta}{\gamma - \alpha \zeta}, 1 \right)
\end{equation}
with 
\begin{equation}
    \alpha_P = \frac{\alpha - \beta}{\gamma - \alpha \zeta}.
\end{equation}
The pseudo-equilibrium is admissible if $0 < \alpha_P < 1$.

\subsection{Existence of limit cycles}
In the case where both equilibria are virtual and stable, we may encounter stable oscillations in the system \cite{simpson2025nonsmooth}. Orbits move towards one virtual equilibrium until crossing the boundary, then start moving towards the other equilibrium until crossing the boundary. This process repeats, leading to a stable limit cycle. Similar behavior was observed in climate and glacier models \cite{morupisi2021analysis}, \cite{walsh2016periodic}, \cite{walsh2020discontinuous}. For our system, the conditions are then
\begin{equation}\label{eqn:ConditionsOsc}
    \frac{\beta}{\alpha} > 1, \quad \frac{\beta + \gamma}{\alpha (\zeta + 1)} < 1.
\end{equation}
In such a case, the boundary consists of the following regions:
\begin{itemize}
    \item crossing region for $N \in \left[0, \frac{\beta}{\alpha} - 1\right)$,
    \item repelling sliding region for $N \in \left(\frac{\beta}{\alpha} - 1, \frac{\beta + \gamma}{\alpha} -1\right)$
    \item crossing region for $N \in \left(\frac{\beta + \gamma}{\alpha} - 1, \infty\right)$.
\end{itemize}
The two tangential singularities are:
\begin{itemize}
    \item $S_1 = \left(\frac{\beta}{\alpha} - 1, 1 \right)$: from \eqref{eqn:second_Lie_minus} we find that the value of $F_1^{2}.H(x)$ in point $S_1$ is $\beta - \alpha > 0$.  Similarly, from \eqref{eqn:second_Lie_plus} we see that $F_2^{2}.H(x)$ at the point $S_1$ is $\beta - \alpha - \alpha \zeta - \beta \gamma < 0$. The tangential singularity is therefore invisible for both $F_1$ and $F_2$.
    \item $S_2 = \left(\frac{\beta + \gamma}{\alpha} - 1, 1 \right)$: from \eqref{eqn:second_Lie_minus} we find that the value of $F_1^{2}.H(x)$ at the point $S_2$ is $\beta^2 + \beta \gamma + \gamma + (\beta - \alpha) > 0.$ From \eqref{eqn:second_Lie_plus} we find that $F_2^{2}.H(x)$ at the point $S_2$ is $\beta + \gamma - \alpha (\zeta + 1) < 0$. The tangential singularity is therefore invisible for both $F_1$ and $F_2$.
\end{itemize}

\begin{figure}
    \centering
    \includegraphics[width=1.0\textwidth]{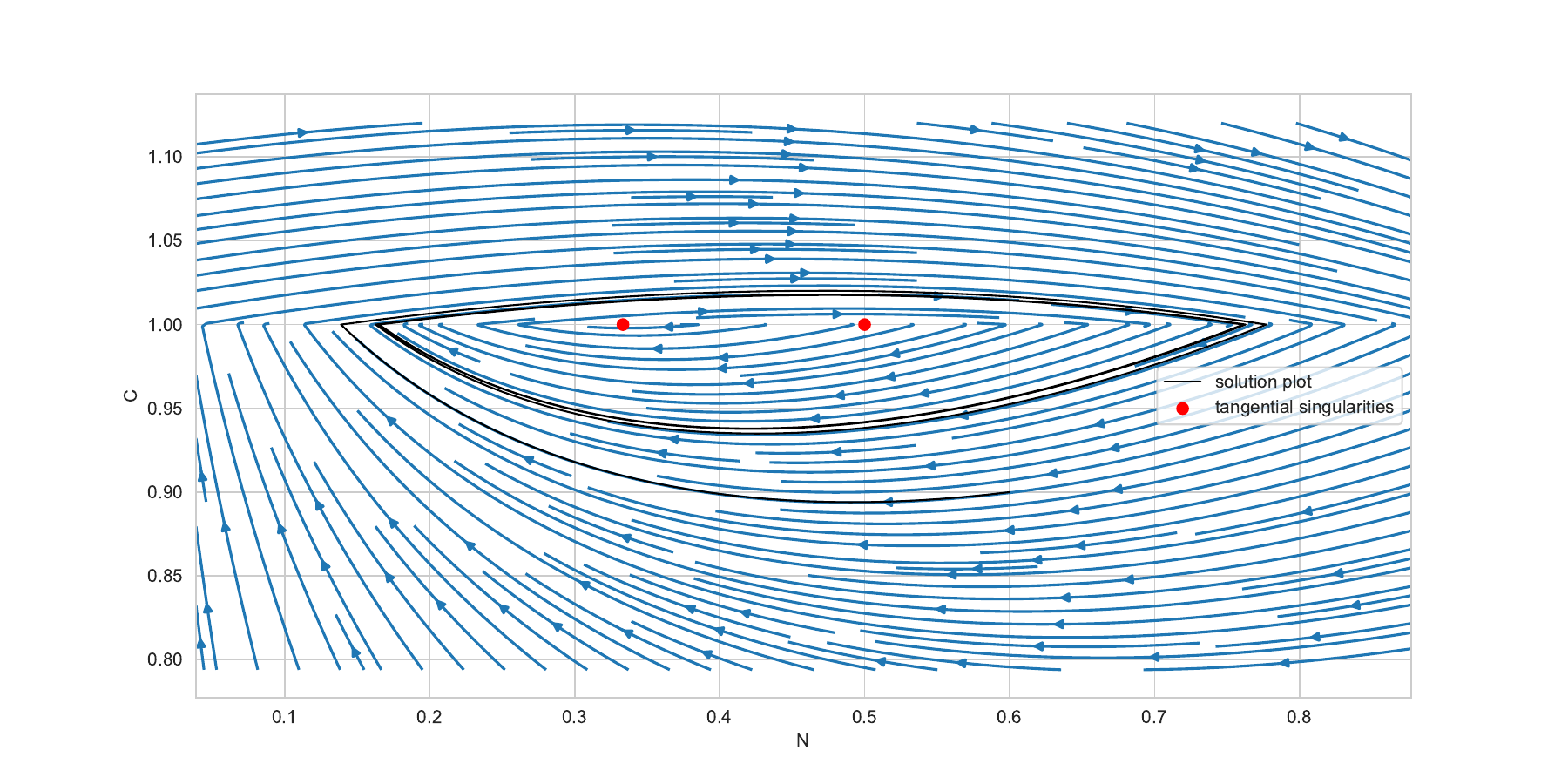}
    \caption{Exemplary phase plane of \eqref{eq:alternative_scaling_system} and its limit cycle.}
    \label{fig:example_phase_plane}
\end{figure}

The exemplary phase plane for the parameters that guaranty the oscillatory behavior of the system is presented in Figure \ref{fig:example_phase_plane}. It is possible to obtain closed form expressions for the flow of our nonsmooth vector field. For $F_1$ we have to solve:
\begin{equation}\label{eqn:cycle_lower_part}
 \frac{dC}{dN} = \frac{-\alpha(N+1)C + \beta}{-N}.
\end{equation}
With an integrating factor of $N^{-\alpha} e^{-\alpha N}$ we obtain
\begin{equation}
 \left( N^{-\alpha} e^{-\alpha N} C\right)' = \beta N^{-\alpha - 1} e^{\alpha N},
\end{equation}
which, after integration, gives:
\begin{equation}\label{eqn:lower_part_general_solution}
    C(N) = \beta (\alpha N)^{\alpha} e^{\alpha N} \Gamma(-\alpha, \alpha N) + K N^{\alpha} e^{\alpha N},
\end{equation}
where $\Gamma(a, x) = \int_{x}^{\infty} t^{a-1} e^{-t} dt$ is the incomplete gamma function and $K$ is the integration constant. With the initial condition $C(N_0) = 1$, we get the specific solution:
\begin{equation}
    C^-(N) = \left( \frac{N}{N_0}\right)^{\alpha} e^{\alpha(N - N_0)} - \beta N^{\alpha} e^{\alpha N} \int_{N_0}^{N} s^{-\alpha - 1} e^{-\alpha s} ds.
\end{equation}
This form can be simplified by changing the integration variable $s\mapsto s/N_0$ to obtain
\begin{equation}\label{eqn:FlowMinus}
    C^-(N) = \left( \frac{N}{N_0}\right)^{\alpha} e^{\alpha(N - N_0)}\left(1 - \beta e^{\alpha N_0} \int_{1}^{\frac{N}{N_0}} s^{-\alpha - 1} e^{-\alpha N_0 s} ds\right).
\end{equation}
Analogously, for $F_2$ we have
\begin{equation}\label{eqn:cycle_upper_part}
    \frac{dC}{dN} = \frac{-\alpha(N+1)C + \beta + \gamma}{-N+\zeta}.
\end{equation}
The solution, given the initial condition $C(N_0) = 1$ can be expressed with the formula:
\begin{equation}\label{eqn:FlowPlus}
    C^+(N) = e^{\alpha(N - N_0)} \left( \frac{\zeta - N}{\zeta - N_0}\right)^{\alpha(\zeta + 1)}\left(1 + (\beta + \gamma) \int_{N_0}^{N} \left(\frac{\zeta - s}{\zeta-N_0}\right)^{-\alpha(\zeta + 1) - 1} e^{-\alpha (s-N_0)} ds\right).
\end{equation}
Based on these two closed forms of the vector flow, we can establish the existence of a limit cycle. 
\begin{theorem}
Assume \eqref{eqn:ConditionsOsc}. Then, the system \eqref{eq:alternative_scaling_system} exhibits a periodic limit cycle. 
\end{theorem}
\begin{proof}
We will construct a Poincar\'e map with the cross-section $C=1$. Pick an initial value $N_0$ with $C(N_0) = 1$ and define the map iteratively as follows. Having a point $N_i$ with $i\in\mathbb{N}$ we define $N_{i+1/2}$ by the $C^-$ flow, that is, $C^-(N_i) = C^-(N_{i+1/2}) = 1$. Next, we map $N_{i+1/2}$ to $N_{i+1}$ with the flow $C^+$ with $C^+(N_{i+1}) = C^-(N_{i+1}) = 1$. That is, $N_i$ and $N_{i+1/2}$ are the initial conditions for the flows \eqref{eqn:FlowMinus} and \eqref{eqn:FlowPlus}, respectively. These conditions are then mapped to $N_{i+1/2}$ and $N_{i+1}$ by respective maps. We write this symbolically as
\begin{equation}
    F(N_i) = N_{i+1}, \quad F(N_0) = N_{1/2}.
\end{equation}
Using both flows \eqref{eqn:cycle_lower_part} and \eqref{eqn:cycle_upper_part}, we will show that the map $F$ has a fixed point. 

First, observe that the singularities of the flow equations are removable. We will show that $C^-(N)$ has a finite limit as $N\rightarrow 0^+$ even though in the governing ODE \eqref{eqn:cycle_lower_part} one has a singularity. Consider the integral in \eqref{eqn:FlowMinus} that can be written as
\begin{equation}
    \int_{1}^{\frac{N}{N_0}} s^{-\alpha - 1} e^{-\alpha N_0 s} ds = - \int_{\frac{N}{N_0}}^1 s^{-\alpha - 1} e^{-\alpha N_0 s} ds \sim - \int_{\frac{N}{N_0}}^1 s^{-\alpha - 1} ds = \frac{1-\left(\frac{N}{N_0}\right)^{-\alpha}}{\alpha} \quad \text{as} \quad N\rightarrow 0^+,
\end{equation}
where the asymptotic equivalence is valid because $e^{-\alpha N_0 s}$ is regular at the origin. Therefore, we have
\begin{equation}
    C^-(N) \sim \left( \frac{N}{N_0}\right)^{\alpha} e^{\alpha(N - N_0)}\left(1 - \frac{\beta}{\alpha} e^{\alpha N_0} \left(1-\left(\frac{N}{N_0}\right)^{-\alpha}\right)\right) \sim \frac{\beta}{\alpha} e^{\alpha N} \sim \frac{\beta}{\alpha}.
\end{equation}
Therefore, we see that $C^-(N) \rightarrow \beta/\alpha$ when $N\rightarrow 0^+$. This result can also be confirmed from equation \eqref{eqn:cycle_lower_part} in which we have a $0/0$ expression that forces the numerator to vanish at the limit of $C=\beta/\alpha$. A completely similar reasoning shows that $C^+(N) \rightarrow (\beta+\gamma)/(\alpha(1+\zeta))$ as $N\rightarrow \zeta^-$.

Now, we will show that by choosing the appropriate value of $N_0$, the Poicar\'e iterations converge to a fixed point. Let $N_0 \geq (\beta+\gamma)/\alpha > \beta/\alpha-1$ be positive by our assumption. Then, by equation \eqref{eqn:cycle_lower_part} we have $dC^-/dN < 0$ at that point. Therefore, initially the concentration of calcium ions decreases monotonically from $C=1$ (until $N = \beta/\alpha -1$ where it attains a minimum). Since we have $C^-(0) = \beta/\alpha > 1$, the function $C(N)$ has to increase through $C=1$. Therefore, there exists a point $0<N_{1/2} < \beta/\alpha-1$ such that $C^-(N_{1/2}) = 1$. Now, the flow $C^+$ is increasing at $N_{1/2}$ due to equation \eqref{eqn:cycle_upper_part}. Hence, we have $C^+ > 1$ in the neighborhood of $N_{1/2}$ and the derivative vanishes at $N = (\beta+\gamma)/\alpha-1$. Because, by the assumption and the exact value of the flow at $N\rightarrow \zeta^{-}$, we have $C^+(\zeta) = (\beta+\gamma)/(\alpha(1+\gamma)) < 1$. Therefore, there exists a point $(\beta+\gamma)/\alpha-1 < N_1 < \zeta$ such that $C^+(N_1) = 1$. This completes one iteration of $F$. Inductively, we define all the following for all $i > 1$. In Fig. \ref{fig:PoincareCrossSection} we have diagrammatically depicted this construction. 

By the construction described above, we see that the compact set $[(\beta+\gamma)/\alpha-1, \zeta]$ is mapped by the continuous function $F$ to itself. Using the standard argument of the intermediate-value theorem for $N \mapsto F(N) - N$ we show that there exists a point $N^*$ such that $F(N^*) = N^*$. Therefore, the Poincar\'e map has a fixed point, and hence there exists a limit cycle of the vector field $(C,N)$. This completes the proof. 
\end{proof}

\begin{figure}
\color{black}
\centering
\begin{tikzpicture}[
    point/.style={circle, fill=black, inner sep=1.2pt},
    font=\small,
    scale = 1.5
    ]
    \draw[->, thick] (-1, 0) -- (8, 0) node[right] {$N$};

    \node (D) at (0,0) [point, label=below:{$0$}]{};
    \node (A) at (2,0) [point, label=below:{$\frac{\beta}{\alpha}-1$}]{};
    \node (B) at (4,0) [point, label=below:{$\frac{\beta+\gamma}{\alpha}-1$}]{};
    \node (C) at (7,0) [point, label=below:{$\zeta$}]{};

    \node (N_half) at (1,0) [point, label=above:{$N_{i+\frac{1}{2}}$}]{};
    
    \node (N_i) at (6,0) [point, label=above:{$N_i$}]{};
    \node (N_i1) at (5,0) [point, label=above:{$N_{i+1}$}]{};
    
    \draw[->, bend left=45, thick] (N_i.south) to node[below, midway] {$C^-$} (N_half.south);

    \draw[->, bend left=45, thick] (N_half.north) to node[above, midway] {$C^+$} (N_i1.north);

\end{tikzpicture}
\caption{A diagram of the Poincar\'e map. }
\label{fig:PoincareCrossSection}
\end{figure}
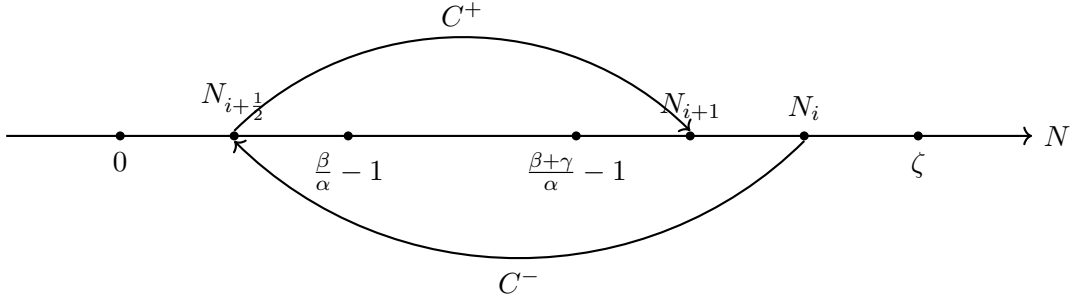

\subsection{Bifurcations}
\begin{definition}\cite{bernardo2008piecewise}
    The piecewise-smooth Filippov system undergoes \textbf{boundary equilibrium bifurcation} at $\mu=\mu^*$ with respect to field $F_i$, $i=1, 2$ if there exists a point $x^*$ such that:
    \begin{enumerate}
        \item $F_i(x^*,\mu^*)=0, but F_j (x^*,\mu^*) \neq 0.$
        \item $H(x^*,\mu^*)=0$.
        \item $F_{i,x}(x^*,\mu^*)$ is invertible.
        \item $H_{\mu}(x^*,\mu^*) - H_x(x^*,\mu^*)[F_{i,x}^{-1} F_{\mu}^{i}](x^*,\mu^*)\neq 0.$
    \end{enumerate}
\end{definition}

\begin{theorem}
    The system \eqref{eq:alternative_scaling_system} undergoes boundary equilibrium bifurcation at $\beta^*=\alpha$ with respect to field $F_1$ for $x^*=\left(0, \frac{\beta}{\alpha} \right)$.
\end{theorem}

\begin{proof}
    Let $x^*=\left(0, \frac{\beta}{\alpha}\right)$ and $\beta^*=\alpha$. 
    Then:
    \begin{enumerate}
        \item $F_1(x^*, \mu^*) = \left(0, -\alpha \frac{\beta}{\alpha} + \beta\right) = 0 $ and $F_2(x^*, \mu^*)=(\zeta, \gamma) \neq 0.$
        \item $H(x^*, \mu^*)=\frac{\beta}{\alpha} - 1 = 1-1 = 0.$
        \item $F_{1,x} = \begin{pmatrix}
        -1 & 0 \\
        -\alpha C & -\alpha N - \alpha
        \end{pmatrix}$ and $\det(F_{1,x})(x^*, \mu^*) = \alpha \neq 0.$
    \item We have: 
    \begin{itemize}
        \item $H_{\mu}(x^*, \mu^*) = 0$,
        \item $H_x(x^*, \mu^*) = [0, 1]^T$,
        \item $F_{1, x}^{-1}(x^*, \mu^*) = \begin{pmatrix}
        -1 & 0 \\
        \frac{\beta}{\alpha} & -\frac{1}{\alpha}
        \end{pmatrix} = \begin{pmatrix}
        -1 & 0 \\
        1 & -\frac{1}{\alpha}
        \end{pmatrix}$
        \item $F_{1, \mu} = [0, 1]^T.$
    \end{itemize}
    Hence:
    \begin{equation}
        H_{\mu}(x^*,\mu^*) - H_x(x^*,\mu^*)[F_{1,x}^{-1} F_{1, \mu}](x^*,\mu^*) =\frac{1}{\alpha} \neq 0.
    \end{equation}
    \end{enumerate}
    Therefore, all the sufficient conditions are satisfied.
\end{proof}

\begin{theorem}
    The system \eqref{eq:alternative_scaling_system} undergoes a boundary equilibrium bifurcation at $\beta^*=\alpha(\zeta+1) - \gamma$ with respect to field $F_2$ for $x^*=\left(\zeta, \frac{\beta + \gamma}{\alpha (\zeta + 1)}\right)$.
\end{theorem}

\begin{proof}
    Let $x^*=\left(\zeta, \frac{\beta + \gamma}{\alpha (\zeta + 1)}\right)$ and $\beta^*=\alpha(\zeta+1) - \gamma$. 
    Then:
    \begin{enumerate}
        \item $F_1(x^*, \mu^*) = \left(-\zeta, -\gamma\right) \neq 0 $ and $F_2(x^*, \mu^*)= 0.$
        \item $H(x^*, \mu^*)=\frac{\beta +\gamma}{\alpha (\zeta + 1)} - 1 = 1-1 = 0.$
        \item $F_{2,x} = \begin{pmatrix}
        -1 & 0 \\
        -\alpha C & -\alpha N - \alpha
        \end{pmatrix}$ and $\det(F_{2,x})(x^*, \mu^*) = \alpha (\zeta + 1) \neq 0.$
    \item We have: 
    \begin{itemize}
        \item $H_{\mu}(x^*, \mu^*) = 0$,
        \item $H_x(x^*, \mu^*) = [0, 1]^T$,
        \item $F_{2, x}^{-1}(x^*, \mu^*) = \begin{pmatrix}
        -1 & 0 \\
        \frac{1}{\zeta + 1} & -\frac{1}{\alpha(\zeta + 1)}
        \end{pmatrix}$
        \item $F_{2, \mu} = [0, 1]^T.$
    \end{itemize}
    Hence:
    \begin{equation}
        H_{\mu}(x^*,\mu^*) - H_x(x^*,\mu^*)[F_{2,x}^{-1} F_{2, \mu}](x^*,\mu^*) = \frac{1}{\alpha(\zeta + 1)} \neq 0.
    \end{equation}
    \end{enumerate}
    All the sufficient conditions are satisfied.
\end{proof}
Let $x$ be a regular equilibrium of $F_1$. We can linearize the system about the boundary equilibrium point, and write for $F_1$:
\begin{equation}
    N(x-x^*) + M(\mu-\mu^*) = 0
\end{equation}
\begin{equation}
    C^T(x-x^*) + D(\mu-\mu^*) = \lambda_1 <0,
\end{equation}
where $N=F_{1, x}$, $M=F_{1, \mu}$, $C^T=H_x$, $D=H_{\mu}$, all evaluated at $x=x^*$, $\mu=\mu^*$. 
Additionally, for pseudo-equilibrium $\tilde{x}$ we have:
\begin{equation}
    N(\tilde{x}-x^*) + M(\mu-\mu^*)+E\alpha_{P} = 0
\end{equation}
\begin{equation}
    C^T(\tilde{x}-x^*)+D(\mu-\mu^*)=0
\end{equation}
\begin{equation}
    \alpha_P > 0,
\end{equation}
where $E=F_2-F_1$ evaluated at $x=x^*$, $\mu=\mu^*$.

\begin{theorem}\cite{bernardo2008piecewise}
\label{thm:bifurcation_type}
For the systems of interest, assuming
\begin{equation}
    \det(N) \neq 0,
\end{equation}
\begin{equation}
    D - C^T N^{-1} M \neq 0,
\end{equation}
\begin{equation}
    C^T N^{-1} E \neq 0.
\end{equation}
\begin{enumerate}
\item \textbf{Persistence} is observed at the boundary equilibrium bifurcation point if
\begin{equation}
    C^T N^{-1} E > 0.
\end{equation}
\item A \textbf{non-smooth fold} is instead observed if
\begin{equation}
    C^T N^{-1} E < 0.
\end{equation}
\end{enumerate}  
\end{theorem}
In our case, we have:
\begin{equation}
    N = \begin{pmatrix}
        -1 & 0 \\
        -\alpha & -{\alpha}
        \end{pmatrix},
\end{equation}
\begin{equation}
    det(N) = \alpha \neq 0,
\end{equation}
\begin{equation}
    D - C^T N^{-1} M = 0 - [0, 1] \begin{pmatrix}
        -1 & 0 \\
        1 & -\frac{1}{\alpha}
        \end{pmatrix} [0, 1]^T = \frac{1}{\alpha} \neq 0,
\end{equation}
\begin{equation}
    C^T N^{-1} E = [0, 1] \begin{pmatrix}
        -1 & 0 \\
        1 & -\frac{1}{\alpha}
        \end{pmatrix} [\zeta, \gamma]^T = \zeta - \frac{\gamma}{\alpha}.
\end{equation}
The bifurcation type therefore depends on the sign of $\zeta - \frac{\gamma}{\alpha}$. If $\gamma < \alpha \zeta$, we observe persistance. If $\gamma > \alpha \zeta$, we observe a non-smooth fold.
We set $\beta$ as the bifurcation parameter. Figure \ref{fig:equilibria_types} presents the types of equilibria for $F_1$ and $F_2$ depending on the parameters $\gamma$ and $\beta$. The red vertical line $\beta=\alpha$ represents the boundary equilibrium of $F_1$. Analogously, the decreasing linear function corresponds to the boundary equilibrium of $F_2$ given by the formula $\beta=\alpha(\zeta+1) - \gamma$.
\begin{figure}
\begin{tikzpicture}[scale=1.2,
    every path/.style={draw=black},
    every node/.style={text=black}
]

\def\xmax{8.5}   
\def\ymax{5.5}   
\def\xalpha{3.2} 

\def\ystart{4.6} 
\def\xint{7.0}   

\fill[red!20]                   
    (\xalpha,0) --             
    (\xalpha, { \ystart*(1 - \xalpha/\xint) }) --  
    (\xint,0) -- cycle;         

\draw[->, thick, draw=black] (0,0) -- (\xmax,0) node[right] {$\beta$};
\draw[->, thick, draw=black] (0,0) -- (0,\ymax) node[above] {$\gamma$};

\node[below, text=black] at (\xmax-1,0) {$\alpha(\zeta + 1)$};
\node[left,  text=black] at (0,\ymax-1) {$\alpha(\zeta + 1)$};

\draw[thick, draw=black] (0,\ystart) -- (\xint,0)
  node[pos=0.55, sloped, above, text=black] {$E^{2}_{\!B}$};

\draw[red, thick] (\xalpha,0) -- (\xalpha,\ymax-0.6);

\node[below, text=red] at (\xalpha,0) {$\alpha$};

\node[text=red] at ({\xalpha+0.35}, {(\ystart + (0 - \ystart)*(\xalpha/\xint)) + 1.75}) {$E^{1}_{\!B}$};

\node[text=black] at (0.5,4.7) {$E^{1}_{\!R}$};
\node[text=black] at (1.0,4.7) {$E^{2}_{\!R}$};

\node[text=black] at (5.1,4.7) {$E^{1}_{\!V}$};
\node[text=black] at (5.6,4.7) {$E^{2}_{\!R}$};

\node[text=black] at (0.5,0.35) {$E^{1}_{\!R}$};
\node[text=black] at (1.0,0.35) {$E^{2}_{\!V}$};

\node[text=black] at (4.5,0.35) {$E^{1}_{\!V}$};
\node[text=black] at (5.0,0.35) {$E^{2}_{\!V}$};
\end{tikzpicture}

\caption{Equilibria types for different values of parameters $\beta$ and $\gamma$.}
    \label{fig:equilibria_types}
\end{figure}
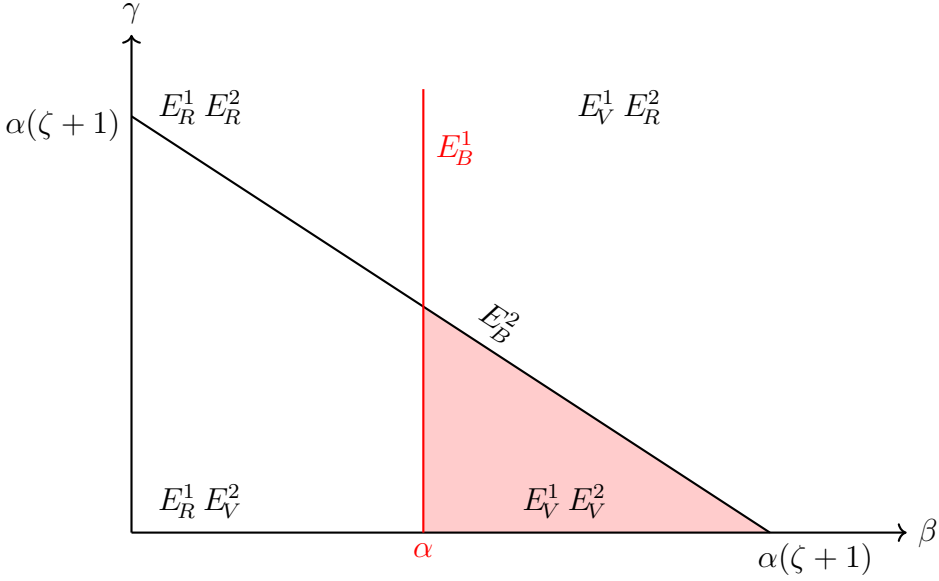
Depending on the value of the parameter $\gamma$, we observe different dynamics of the system. In our analysis, we set $\alpha=0.3$ and $\zeta=2$. When $0 < \gamma < \alpha \zeta$ (Figure \ref{fig:bifuractions_0.5}), for small values of $\beta$ we observe a regular equilibrium of $F_1$ coexisting with a virtual equilibrium of $F_2$. Increasing $\beta$ leads to boundary crossing and then two virtual equilibria. It represents the oscillation case. Then the boundary equilibrium of $F_2$ is reached and finally we observe the virtual equilibrium of $F_1$ and the regular equilibrium of $F_2$. 
\begin{figure}
    \centering
    \includegraphics[width=\linewidth]{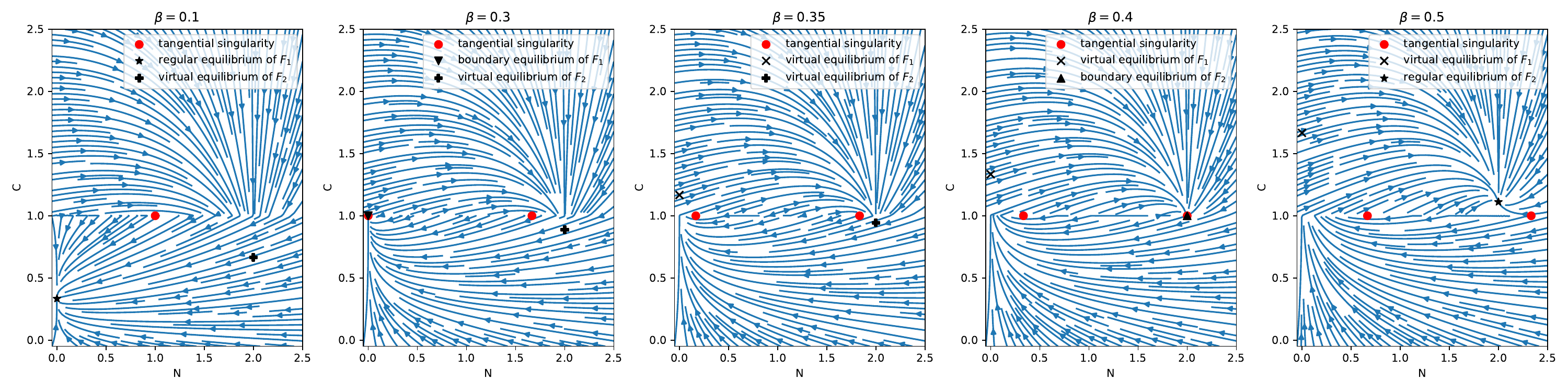}
    \caption{Bifurcations observed for $\gamma=0.5$ and varying values of parameter $\beta$.}
    \label{fig:bifuractions_0.5}
\end{figure}
 When $\gamma=\alpha \zeta$ we have a unique situation when the regular equilibrium of $F_1$ and the virtual equilibrium of $F_2$ change to the regular equilibrium of $F_2$ and the virtual equilibrium of $F_1$. In the intermediate phase, for $\beta=\alpha$, we have two boundary equilibria that exist together. 
\begin{figure}
    \centering
    \includegraphics[width=\linewidth]{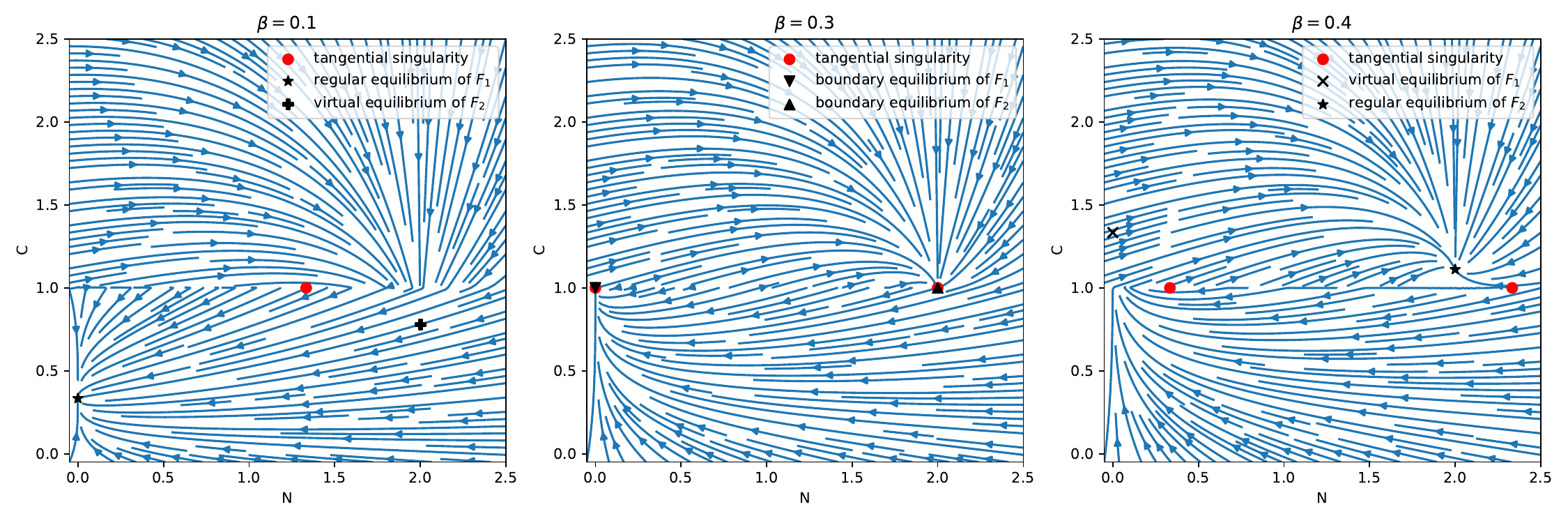}
    \caption{Bifurcations observed for $\gamma=0.6$ and varying values of parameter $\beta$.}
    \label{fig:bifuractions_2}
\end{figure}
For $\alpha \zeta < \gamma < \alpha(\zeta + 1)$ we again observe the transition from the regular equilibrium of $F_1$ and the virtual equilibrium of $F_2$ to the regular equilibrium of $F_2$ and the virtual equilibrium of $F_1$. Compared to the considerations for $\gamma < \alpha \zeta$, when two virtual equilibria coexisted, we had two regular equilibria in the intermediate stage. 
\begin{figure}
    \centering
    \includegraphics[width=\linewidth]{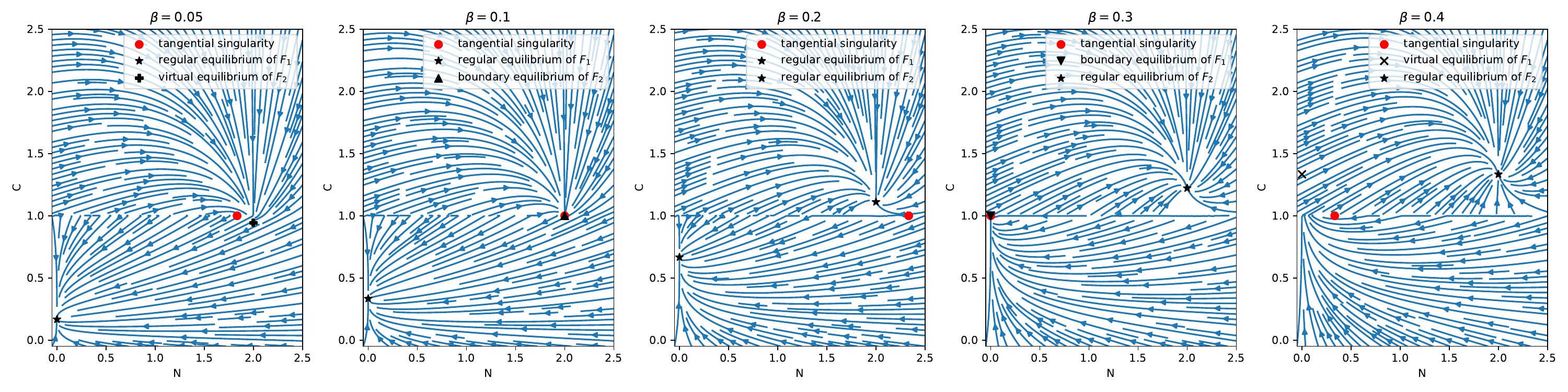}
    \caption{Bifurcations observed for $\gamma=0.8$ and varying values of parameter $\beta$.}
    \label{fig:bifuractions_2.5}
\end{figure}
Finally, for $\alpha (\zeta + 1) < \gamma$ and small values of $\beta$, we observe two regular equilibria. Then, through BEB, equilibrium of $F_1$ becomes vitual. 
\begin{figure}
    \centering
    \includegraphics[width=\linewidth]{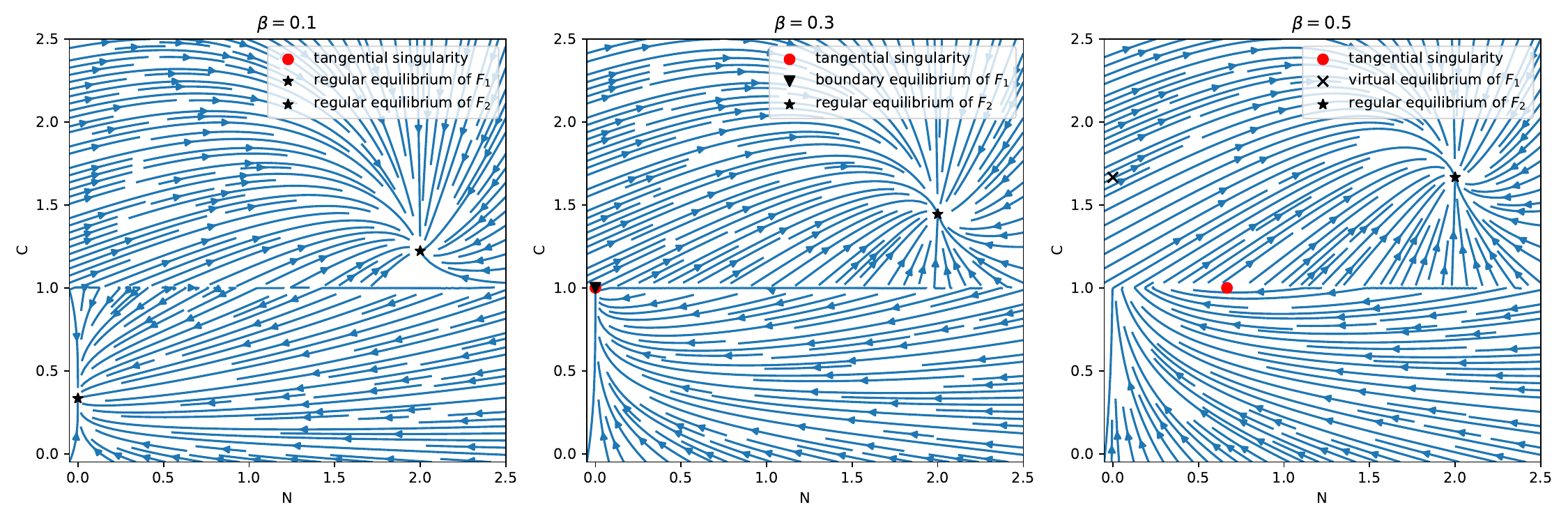}
    \caption{Bifurcations observed for $\gamma=1$ and varying values of parameter $\beta$.}
    \label{fig:bifuractions_3.5}
\end{figure}

\section{Conclusion and future work}
In this work, we have established a comprehensive mathematical framework that unifies the hemodynamic and lymphodynamic transport phenomena within biological vessels. By rigorously deriving the governing partial differential equations through asymptotic perturbation methods, we provided a fluid-mechanical description that accounts for the distinct rheological properties of both blood and lymph. A central contribution of this study is the integration of active biochemical regulation into the fluid dynamic model. By coupling the PDE system with non-smooth ordinary differential equations describing the kinetics of calcium ions ($Ca^{2+}$) and nitric oxide ($NO$), we captured the feedback loops inherent in lymphangion valve activation. We identified specific parameter regimes where the system exhibits limit cycles. The existence of these oscillations in the non-smooth formulation shows the robustness of the physiological mechanism and provides a deterministic explanation for the rhythmic contractions observed in vivo. Ultimately, this framework serves as a foundational step toward a more predictive and mechanically grounded understanding of the lymphatic system.

Future work will focus on developing an analysis of the coupled flow-biochemistry model along with designing efficient numerical methods to solve the main nonlinear (and possibly) nonlocal system of differential equations. 

\bibliographystyle{plain} 
\bibliography{bibliography}

\end{document}